\documentclass[leqno,a4paper]{article}

\usepackage{amscd}
\usepackage{amsmath}
\usepackage{amsthm}
\usepackage{amssymb}
\usepackage[english]{babel}
\usepackage{bbm}
\usepackage{enumerate}
\usepackage{enumitem}
\usepackage{epsfig}
\usepackage{euscript}
\usepackage{graphicx}
\usepackage{graphics}
\usepackage{mleftright}
\usepackage{multicol}
\usepackage{subcaption}
\usepackage{tabularx}
\usepackage{xspace}

\newlist{steps}{enumerate}{1}
\setlist[steps, 1]{label = Step \arabic*:}
\usepackage[usenames,dvipsnames,svgnames,table]{xcolor}

\graphicspath{{/home/isaac/LaTeXFiles/Niall-Romina-Shari-Isaac/images/}}

\newcommand{\EB}[1]{E_{#1} \times B_{0}}

        \newtheorem{lemma}{Lemma}[section]
        \newtheorem{proposition}[lemma]{Proposition}
        \newtheorem{theorem}[lemma]{Theorem}
        \newtheorem{definition}{Definition}[section]

\numberwithin{equation}{section}

\title{\bf{Stability and reconstruction of a special type of anisotropic conductivity in magneto-acoustic tomography with magnetic induction}}
\author{Niall Donlon\thanks{Department of Mathematics and Statistics, University of Limerick, Ireland.  Email: niall.donlon@ul.ie}\qquad Romina Gaburro\thanks{Department of Mathematics and Statistics, Health Research Institute (HRI), CONFIRM, University of Limerick, Ireland.  Email: romina.gaburro@ul.ie}\\ Shari Moskow\thanks{Department of Mathematics, Drexel University, Philadelphia, PA, USA. Email: slm84@drexel.edu}\qquad Isaac Woods\thanks{Department of Mathematics, Drexel University, Philadelphia, PA, USA. Email: iw76@drexel.edu}}
\date{}

\begin{document}
\maketitle

\begin{abstract}
We consider the issues of stability and reconstruction of the electrical anisotropic conductivity of biological tissues in a domain $\Omega\subset\mathbb{R}^3$ by means of the hybrid inverse problem of magneto-acoustic tomography with magnetic induction (MAT-MI). The class of anisotropic conductivities considered here is of type $\sigma(\cdot)=A(\cdot,\gamma(\cdot))$ in $\Omega$, where $[\lambda^{-1}, \lambda]\ni t\mapsto A(\cdot, t)$ is a one-parameter family of matrix-valued functions which are \textit{a-priori} known to be $C^{1,\beta}$, allowing us to stably reconstruct $\gamma$ in $\Omega$ in terms of an internal functional $F(\sigma)$. Our results also extend previous results in MAT-MI where $\sigma(\cdot) = \gamma(\cdot) D(\cdot)$, with $D$ an \textit{a-priori} known matrix-valued function on $\Omega$ to a more general anisotropic structure which depends non-linearly on the scalar function $\gamma$ to be reconstructed.
\end{abstract}



\section{Introduction}\label{sec1}
\setcounter{equation}{0}
In many physical situations, it is important to determine certain physical properties of the interior of a body that cannot be measured directly. Inverse problems are employed to infer such information from external observations. If one is interested in imaging biological tissue inside the human body, its electrical conductivity distribution can provide valuable information about its state. A recently developed non-invasive imaging modality based on the determination of the electrical conductivity distribution of biological tissue is Magnetoacoustic Tomography with Magnetic Induction (MAT-MI).

If a biological conductive body, modelled by a domain $\Omega\subset\mathbb{R}^3$ with smooth boundary $\partial\Omega$, is placed in a static magnetic field $B_0 = (0, 0, 1)$, it starts to emit ultrasound waves that can be measured around the body, i.e. at a series of locations on $\partial\Omega$. To be more precise, to cause an eddy current within the conductive tissue, a pulsed time-dependent magnetic field  $B(t)$ is applied to excite the tissue in $\Omega$, which, in turn, emits the ultrasound waves. For the sake of simplicity and to follow the research line already initiated in \cite{A-Q-S-Z}, \cite{Q-S}, we assume here that $B(t)$ is time-independent and that $B(t) = B_0$. MAT-MI is an example of \textit{hybrid inverse problems}, which typically combine a high contrast modality with a high resolution
one and they typically involve two steps. In the first step, a well-posed problem involving high resolution and low contrast modality is solved from the knowledge of boundary measurements. In the second step, a quantitative reconstruction of the parameters of interest (describing a physical property of the medium in question) is solved by knowledge of a so-called \textit{internal functional} which has
been reconstructed during the first step. Hybrid methods are mainly concerned with the solution of the second step (assuming that the first step has been successfully performed). For a review on hybrid imaging modalities we refer to \cite{Ba}.

The first step in MAT-MI is to retrieve the acoustic source from the measurements around the object in the scalar wave equation. Then, in the second step, MAT-MI reconstructs the distribution of electrical conductivity from acoustic source information (see \cite{A-Q-S-Z}, \cite{Q-S}). This second step is the focus of this paper, where we study, in the MAT-MI experiment, the issues of stability and reconstruction for a special type of anisotropic conductivity $\sigma$ in terms of internal measurements of the acoustic sources (which are assumed to be known after the first step has been performed) modelled by the internal functional
 \begin{equation}\label{functional}
 F(\sigma) =\nabla \cdot (\sigma E_{\sigma} \times B_0),\qquad\textnormal{in}\quad \Omega,
 \end{equation}
 where $E_{\sigma}$ solves the Maxwell's equations
 \begin{equation}\label{Maxwell}
      \Bigg\{\begin{array}{llll}
         \nabla \times E_{\sigma} & = B_0, \quad &\text{in} \quad \Omega,\\
        \nabla \cdot \sigma E_{\sigma} & = 0, \quad &\text{on} \quad  \Omega,\\        \sigma E_{\sigma} \cdot \nu & = 0, \quad &\text{on} \quad \partial \Omega
        \end{array}
         \end{equation}
and $\sigma\in C^{1,\beta}(\overline\Omega)$ is symmetric on $\Omega$ and satisfies a uniform ellipticity condition (the precise formulation of the problem is given in section \ref{main assumptions}).

For clinical and research purposes, the electrical conductivity of biological tissues can provide valuable information as the conductivity varies significantly within the human body. Other and more established medical imaging modalities, like Computerized Tomography (CT), Magnetic Resonance Imaging (MRI) and ultrasound imaging, are typically capable of creating images of the human body with very high resolution. These modalities often fail to exhibit a sufficient contrast between different type of tissues. Imaging modalities like Optical Tomography (OT) and Electrical Impedance Tomography (EIT) do, on the other hand, display such high contrast at the expense of a poor resolution (\cite{Ar}, \cite{Bo}, \cite{U}). EIT, similarly to MAT-MI, also provides information about materials and biological tissues in terms of their conductivity. The resulting images from measurements of electrostatic voltages and current flux taken on the surface of the body under investigation, are often blurred, due to the EIT poor resolution. MAT-MI, on the other hand, has the potential to overcome this issue by providing images of the conductivity of a body in terms of internal measurements that have been obtained by means of a high resolution imaging modality performed in the first step of MAT-MI. Hence MAT-MI has the potential to provide high contrast images of a body in terms of its internal conductivity that also benefits from a reasonably good resolution.

Biological tissues are known to have a have anisotropic conductivity \cite{Mar}. In EIT, since its first mathematical formulation by A. Calder\'on in his 1980 seminal paper \cite{C}, a lot of progress has been made, but the problem of uniquely determining the anisotropic conductivity of a body by means of EIT measurements is still considered an open problem. Partial results of uniqueness and stability for this inverse problem have been obtained in \cite{Al}, \cite{Al-Cab}, \cite{Al-dH-G}, \cite{Al-dH-G-S}, \cite{Al-G}, \cite{Al-G1}, \cite{As-La-P}, \cite{Be}, \cite{F-G-S}, \cite{G-Li}, \cite{G-S} \cite{Gr-La-U1}, \cite{Ko-V1}, \cite{La-U}, \cite{La-U-T}, \cite{Le-U}, \cite{Li}, \cite{Sy}.

In the MAT-MI experiment, results of stability and reconstruction of the conductivity $\sigma$ in terms of the internal functional $F(\sigma)$ given in \eqref{functional}, have been obtained in  \cite{Q-S} and \cite{A-Q-S-Z}, where the isotropic case $\sigma=\gamma I$ (here $I$ denotes the $3\times 3$ identity matrix and $\gamma$ is a positive scalar function on $\Omega$) and the special anisotropic case $\sigma = \gamma D$, with $D$ a known $3\times 3$ symmetric matrix and $\gamma$ is a positive scalar function on $\Omega$, were considered, respectively. In the present paper, we extend these results to the case where the anisotropic conductivity $\sigma$ is of type $\sigma= A(\cdot,\gamma(\cdot))$, and $A(\cdot, t)$ is known for $t\in[\lambda^{-1}, \lambda]$.

More precisely, we start by considering the simpler case where $\sigma$ is a symmetric, uniformly positive definite matrix which is \textit{a-priori} known to have the structure $\sigma(x) = A(\gamma(x))$, $x\in\overline{\Omega}$, where the one-parameter family of matrix-valued functions
\[t\mapsto A(t),\qquad t\in[\lambda^{-1}, \lambda]\]
is assumed to be known and to belong to a certain class $\mathcal{A}$ defined below (Definition \ref{A}) and $\gamma = \gamma(x)$, $x\in\overline\Omega$ is an unknown scalar function to be determined.
The above structure for $\sigma$ is also generalized to the case $\sigma = A(x, \gamma(x))$, $x\in\overline{\Omega}$, where the one-parameter family of matrix-valued functions
\[t\mapsto A(x,\:t),\qquad\textnormal{for\:any}\quad x\in\overline\Omega,\qquad t\in[\lambda^{-1}, \lambda]\]
is assumed to be known and to belong to a certain class $\mathcal{A'}$ defined below (Definition \ref{A'}) and $\gamma = \gamma(x)$, $x\in\overline\Omega$ is again an unknown scalar function to be determined. The latter, in particular, extends the results in \cite{Q-S} and \cite{A-Q-S-Z}, where the problem of stability and reconstruction in MAT-MI has been addressed in the isotropic case and in the anisotropic case $\sigma(x)=\gamma(x)D(x)$, where $D$, in this case, is a known matrix-valued function and $\gamma= \gamma(x)$, $x\in\overline{\Omega}$ is the unknown scalar function to be determined. For both the cases when $A\in \mathcal{A}$ and $A\in \mathcal{A'}$ considered in this manuscript, we prove that we can stably reconstruct the scalar function $\gamma$ in terms of $F$.

In this paper we also present several numerical experiments in which we reconstruct the scalar function $\gamma=\gamma(x)$, $x\in\Omega$ for both types of anisotropic conductivities $\sigma(x) = A(\gamma(x))$ and $\sigma(x)= A(x,\gamma(x))$, $x\in\Omega$. We start by considering a number of examples in which $A\in\mathcal{A}$ only (Examples 1 - 4). The more general case of $A\in\mathcal{A'}$ is exploited in Examples 5 and 6. The reconstruction is based on an algorithm introduced in \cite{Q-S}, \cite{A-Q-S-Z}, which relies on projecting its iterates into a convex subset of $C^{1,\beta}(\Omega)$ in order to restore well-posedness in the inverse problem. We note that in the particular examples considered here, the projection step seems unnecessary for the convergence of the iterates. Moreover, in two of our examples, the true conductivity $\gamma_*$ we wish to reconstruct is not $C^1$, hence is less regular than the classes of conductivities we considered in our theoretical framework. This not only allows us to test the performance of the reconstruction algorithm on non-smooth data, but it also provides insights about possibly lowering the regularity of the conductivity considered in our theoretical framework (the stability estimates in \ref{main assumptions} and the convergence analysis we carry out in \ref{convergence analysis}). It will be part of future work to extend our theoretical framework to a class of conductivities with lower regularity assumptions and a more general anisotropic structure. The anisotropic structures considered in the current paper are one-parameter families of symmetric and uniformly positive definite matrices that depend on the unknown scalar function $\gamma$ in a nonlinear way, extending the results in \cite{A-Q-S-Z} to the case of a more realistic dependence of the anisotropic structure on the unknown scalar function $\gamma$. We wish to stress that this paper aims at providing a first step in the treatment in MAT-MI of anisotropic structures that depend nonlinearly on (possibly) a finite number of unknown scalar functions to be stably reconstructed in the MAT-MI experiment. It will be the subject of future work to consider the fully anisotropic case and to which extent the MAT-MI experiment can be employed to determine the anisotropic structure itself.

The paper is organized as follows. In Section \ref{main assumptions} we introduce notation, state our main assumptions, and define precisely the two classes $\mathcal{A}$ and $\mathcal{A}'$ for $A$, corresponding to both the simpler and spatially dependent anisotropic case, respectively. The main stability results are also contained in this section. Section \ref{proofs stability estimates} contains the proofs of our main stability results, and includes estimates for the electromagnetic boundary value problem. In Section \ref{reconstruction} we describe the reconstruction algorithm (Section \ref{algorithm}) and prove convergence results (Section \ref{convergence analysis}), again corresponding to both classes of $A$. Section \ref{numerics} contains several numerical experiments demonstrating accurate reconstructions, including two cases of non-smooth true scalar functions $\gamma_*$ and one fully three dimensional reconstruction. Some final concluding remarks are included in Section \ref{conclusions}.


\section{Main assumptions and stability estimates}\label{main assumptions}

Throughout this paper the medium to be imaged is a bounded domain $\Omega$ in $\mathbb{R}^3$ with $C^{1,\beta}$ boundary $\partial\Omega$ (see definition \ref{boundary}) and diameter $\textnormal{diam}(\Omega)= M$, for some constant $M>0$.  For a point $x\in\mathbb{R}^3$, we will denote $x=(x', x_3)$, where $x' \in\mathbb{R}^2$ and $x_3\in\mathbb{R}$. We will also denote by $B_r$, $B'_r$ the open balls $B_r(0)$ and $B'_r(0)$ in $\mathbb{R}^3$ and $\mathbb{R}^2$, respectively and by $Q_r$ the cylinder $Q_r(0)$ in $\mathbb{R}^3$ defined by
\[Q_r(0) = B'_r \times (x_3 - r, x_3 + r).\]
From here onwards we fix a number $\beta$ satisfying $0<\beta\leq 1$.

\begin{definition}\label{boundary}
Given a bounded domain $\Omega\subset\mathbb{R}^{3}$, we say that $\partial\Omega$ is of class $C^{1,\beta}$ with constants $r_0 , L>0$, if for every $P\in\partial\Omega$, there exists a rigid transformation of coordinates under which $P=0$ and
\[\Omega\cap Q_{r_0} = \left\{(x', x_n)\in Q_{r_0}\quad |\quad x_3>\varphi(x')\right\},\]
where $\varphi$ is a $C^{1,\beta}$ function on $B'_{r_0}$ satisfying
\[\varphi(0) = |\nabla\varphi(0)| = 0\]
and
\[||\varphi||_{C^{1,\beta}(B'_{r_0})}\leq Lr_0,\]
where we use the normalization convention
\[||\varphi||_{C^{1,\beta}(B'_{r_0})} = ||\varphi||_{L^{\infty}(B'_{r_0})} + r_0 ||\nabla\varphi||_{L^{\infty}(B'_{r_0})} + r_0^{1+\beta} \sup_{x', y'\in B'_{r_0},\:
x'\neq y'} \frac{|\nabla\varphi(x') - \nabla\varphi(y')|}{|x' - y'|^{\beta}}.\]
\end{definition}
Once for all it is understood that $\Omega\subset\mathbb{R}^3$ is a bounded domain with diameter $\textnormal{diam}(\Omega)= M$ and $C^{1,\beta}$ boundary $\partial\Omega$. We define below the two classes of admissible anisotropic structures considered in this paper.

\begin{definition}\label{A}
Given  $\Lambda$, $\mathcal{E}_1 >0$ and denoting by $Sym$ the class of $3\times 3$ real-valued symmetric matrices, we say that $A(\cdot)\in\mathcal{A}$ if the following conditions hold:
\begin{equation}\label{gamma holder}
  A\in C^{1,\beta}([\lambda^{-1}, \lambda], Sym),
\end{equation}
 \begin{equation} \label{Derivative assumption}
  \begin{split}
    &\underset{t \in [\lambda^{-1}, \lambda]}{\sup}  \Big|\frac{d A}{d t}(t)\Big| + \underset{t_1\neq t_2}{\underset{t_1,\: t_2 \in [\lambda^{-1}, \lambda]}{\sup}} \frac{|\frac{d A}{d t_i}(t_1) - \frac{d A}{d t_i}(t_2)|}{|t_1 - t_2|^\beta} \leq \mathcal{E}_1
\end{split}
\end{equation}
and
\begin{equation}\label{ellipticity 1}
    \Lambda^{-1} |\xi|^2 \leq A(t) \xi \cdot \xi \leq \Lambda |\xi|^2, \quad \textnormal{for\:every} \quad t \in [\lambda^{-1}, \lambda], \quad \xi \in \mathbb{R}^3.
\end{equation}
\end{definition}

\begin{proposition}\label{proposition2.1}
If $A\in\mathcal{A}$ and $\gamma$ satisfies
\begin{eqnarray}\label{assumptions sigma}
 \lambda^{-1}\leq\gamma(x) &\leq &\lambda,\quad\textnormal{for\:almost\:every}\:x\in\Omega\label{sigma condition 1} ,\\
  ||\gamma||_{C^{1, \beta} (\Bar{\Omega})} &\leq &K\label{sigma condition 2},
  \end{eqnarray}
for some constants $\lambda , K>0$, then we have
\begin{equation}
    A(\gamma(\cdot)) \in C^{1, \beta} (\Bar{\Omega},Sym_n)
\end{equation}
and furthermore,
\begin{equation}
    ||A(\gamma(\cdot))||_{C^{1, \beta} (\Bar{\Omega})} \leq \Lambda + K \mathcal{E}_1 M(1 + 2M^{\beta}) := \mathcal{F}_1,
\end{equation}
where $M$, $\mathcal{E}_1$ and $\Lambda$ have been introduced above.
\end{proposition}

\begin{proof}
The proof is a straightforward consequence of equality
\begin{equation}
    \begin{split}
         ||A( \gamma(\cdot))||_{C^{1,\beta}(\Bar{\Omega})}
         &= ||A(\gamma(\cdot))||_{L^{\infty}(\Bar{\Omega})}\\
         &+ M \underset{1 \leq i \leq n}{\sup} \Big|\Big|\frac{\partial A}{\partial x_i}(\gamma(\cdot))\Big|\Big|_{L^\infty(\Bar{\Omega})}\\
         &+ M ^{1+\beta} \underset{1 \leq i \leq n}{\sup} \quad \underset{x \neq y}{\underset{x,y \in \Bar{\Omega}}{\sup}} \frac{|\frac{\partial A}{\partial x_i}(\gamma(x)) - \frac{\partial A}{\partial x_i}(\gamma(y))|}{|x - y|^\beta}
    \end{split}
\end{equation}
and our assumptions.
\end{proof}


\begin{definition}\label{A'}
Given $\Lambda$, $\mathcal{E}_2 >0$, we say that $A(\cdot,\cdot)\in\mathcal{A'}$ if the following conditions hold:
\begin{equation}\label{gamma holder 2}
  A\in C^{1,\beta}(\Bar\Omega\times [\lambda^{-1}, \lambda], Sym),
\end{equation}
 \begin{equation}\label{partial derivative assumption}
    \begin{split}
    &\underset{t \in [\lambda^{-1}, \lambda]}{\sup} \Bigg\{\quad \underset{1\leq i \leq n}{\sup} \Big|\Big|\frac{\partial A}{\partial x_i}(\cdot, t)\Big|\Big|_{L^\infty(\Bar{\Omega})} + \Big|\Big|\frac{\partial A}{\partial t}(\cdot, t)\Big|\Big|_{L^\infty(\Bar{\Omega})} \\
    &+ \underset{1\leq i \leq n}{\sup} \quad \underset{x\neq y}{\underset{x,y \in \Bar{\Omega}}{\sup}}  \frac{|\frac{\partial A}{\partial x_i}(x, t) - \frac{\partial A}{\partial x_i}(y, t)|}{|x-y|^\beta} + \underset{x\neq y}{\underset{x,y \in \Bar{\Omega}}{\sup}}  \frac{|\frac{\partial A}{\partial t}(x, t) - \frac{\partial A}{\partial t}(y, t)|}{|x-y|^\beta} \Bigg\} \\
&+ \underset{t_1 \neq t_2}{\underset{t_1,\:t_2 \in [\lambda^{-1}, \lambda]}{\sup}} \Bigg\{ \frac{||\frac{\partial A}{\partial x_i}(\cdot, t_1) - \frac{\partial A}{\partial x_i}(\cdot, t_2)||_{L^\infty(\Bar{\Omega})}}{|t_1 - t_2|^\beta} + \frac{||\frac{\partial A}{\partial t}(\cdot, t_1) - \frac{\partial A}{\partial t}(\cdot, t_2)||_{L^\infty(\Bar{\Omega})}}{|t_1 - t_2|^\beta} \Bigg\}\leq \mathcal{E}_2
\end{split}
\end{equation}
and
\begin{equation}\label{ellipticity 2}
    \Lambda^{-1} |\xi|^2 \leq A(x,t) \xi \cdot \xi \leq \Lambda |\xi|^2, \quad \textnormal{for\:a.e.}\: x \in \Bar{\Omega}, \quad \forall t \in [\lambda^{-1}, \lambda], \quad \forall \xi \in \mathbb{R}^3.
\end{equation}
\end{definition}

We observe that \eqref{ellipticity 1}, \eqref{ellipticity 2} are conditions of uniform ellipticity.

\begin{proposition}
If $A\in\mathcal{A'}$ and $\gamma$ satisfies \eqref{sigma condition 1}, \eqref{sigma condition 2}, then we have
\begin{equation}
    ||A(\cdot, \gamma(\cdot))||_{C^{1, \beta} (\Bar{\Omega})} \leq \Lambda + M \mathcal{E}_2(1+K)(1 + 2M^\beta) := \mathcal{F}_2,
\end{equation}
\end{proposition}
where where $M$, $\mathcal{E}_2$ and $\Lambda$ have been introduced above.

\begin{proof}
This is a straightforward consequence of equalities
\begin{equation}
    \frac{\partial A}{\partial x_i}( x, \gamma(x)) = \frac{\partial A}{\partial x_i}( x, t)\big{|}_{(x,t)=(x,\gamma(x))} + \frac{\partial A}{\partial t}( x, t)\big|_{(x,t)=(x,\gamma(x))} \frac{\partial \gamma}{\partial x_i}(x),
\end{equation}
together with our assumptions.
\end{proof}
Below are our two stability estimates of $\gamma$ in terms of the internal functional $F$ for the two classes of admissible anisotropic structures $\mathcal{A}$ and $\mathcal{A'}$, respectively.
\begin{theorem}\label{theorem 1}
Let $A\in\mathcal{A}$ and assume $\gamma_i$ satisfies \eqref{sigma condition 1}, \eqref{sigma condition 2}, for $i=1,2$,  $\gamma_1 = \gamma_2$ on $\partial\Omega$ and
\begin{eqnarray}\label{Gradient sigma difference assumption}
||\nabla(\gamma_1 - \gamma_2)||_{L^2(\Omega)} &\leq \mathcal{C} ||\gamma_1 - \gamma_2||_{L^2(\Omega)},
\end{eqnarray}
for some constant $\mathcal{C}>0$. Denoting
\begin{equation}\label{Fi 1}
F_i := \nabla \cdot (A(\gamma_i) E_{A(\gamma_i)}\times B_0),\qquad\textnormal{for}\quad i=1,2,
\end{equation}
if $\mathcal{C}\Big(C_1(\mathcal{E}_1 +1) + \Lambda C_3\Big)<\frac{1}{2}$, where $C_1 = (1 + \Lambda \mathcal{F}_1)(\frac{1}{2} M (\mu (\Omega))^\frac{1}{2})$ and $C_3 = \Lambda \mathcal{E}_1 C_1$, then
\begin{equation}
    ||\gamma_1 - \gamma_2||_{L^2(\Omega)} \leq C ||F_1 - F_2||_{L^2(\Omega)},
\end{equation}
where $C>0$ is a constant that depends on $\mathcal{C}$, $K$, $\mathcal{E}_1$, $\Lambda$, $M$, $r_0$ and $L$ only.
\end{theorem}
More generally, we obtain the following main stability result.
\begin{theorem}\label{theorem 2}
Let $A\in\mathcal{A}'$ and $\gamma_i$ be as in Theorem \ref{theorem 1}, for $i=1,2$. Denoting
\begin{equation}\label{Fi 2}
\tilde{F}_i := \nabla \cdot (A(\cdot,\:\gamma_i) E_{A(\cdot,\:\gamma_i)}\times B_0),\qquad\textnormal{for}\quad i=1,2,
\end{equation}
if $\mathcal{C}\Big(C_2(\mathcal{E}_2 +1) + \Lambda C_4\Big)<\frac{1}{2}$, where $C_2 = (1 + \Lambda \mathcal{F}_2)(\frac{1}{2} M (\mu (\Omega))^\frac{1}{2})$ and $C_4 = \Lambda \mathcal{E}_2 C_2$, then
\begin{equation}
    ||\gamma_1 - \gamma_2||_{L^2(\Omega)} \leq \Tilde{C} ||\tilde{F}_1 - \tilde{F}_2||_{L^2(\Omega)},
\end{equation}
where $\Tilde{C}>0$ is a constant that depends on $\mathcal{C}$, $K$, $\mathcal{E}_2$, $\Lambda$, $M$, $r_0$ and $L$ only.
\end{theorem}


\section{Proof of the stability estimates}\label{proofs stability estimates}

To prove our main stability results, we proceed by reducing the Maxwell system \eqref{Maxwell} into a Neumann boundary value problem, for which we recall standard estimates that will be needed for the arguments in the proofs of our stability results.


\subsection{Technical results}

In what follows we recall standard estimates for the solution of Neumann boundary value problems of type
\begin{equation}\label{e1}
      \Bigg\{\begin{array}{llll}
        \nabla \cdot (\sigma \nabla u) & = -\nabla \cdot E, \quad &\text{in} \quad  \Omega,\\
        (\sigma \nabla u + E) \cdot \nu & = 0, \quad &\text{on} \quad \partial \Omega.
        \end{array}
        \end{equation}
        Such estimates are standard in the theory of elliptic partial differential equations and we recall them, together with their proof, for the sake of completeness.
\begin{proposition} \label{proposition3.3}
Let $\sigma\in C^{1,\beta}(\overline\Omega)$ be a matrix-valued function satisfying the uniform ellipticity condition
\begin{equation}\label{ellipticity 3}
    \Lambda^{-1} |\xi|^2 \leq \sigma(x) \xi \cdot \xi \leq \Lambda |\xi|^2, \quad \textnormal{for\:a.e.}\: x \in \Bar{\Omega},\quad \forall \xi \in \mathbb{R}^3
\end{equation}
and let $E\in$ $L^2(\Omega)$ be a vector-valued function. The Neumann problem \eqref{e1} has a unique solution $u \in H^1(\Omega)$ (up to an additive constant) that satisfies
\begin{equation} \label{Delta u proposition}
    ||\nabla u||_{L^2(\Omega)} \leq \Lambda ||E||_{L^2(\Omega)}.
\end{equation}
\end{proposition}

\begin{proof}
As a straighforward consequence of
\begin{equation} \label{u12}
\int_\Omega A \nabla{u} \cdot \nabla u\:dx = - \int_\Omega {E}\cdot\nabla{u}\:dx,
\end{equation}
we obtain
\begin{eqnarray}
\Lambda^{-1} ||\nabla{u}||^2_{L^2(\Omega)} \leq \Big| - \int_\Omega {E}\cdot\nabla{u} dx \Big| \leq ||{E}||_{L^2(\Omega)} ||\nabla{u}||_{L^2(\Omega)},
\end{eqnarray}
which concludes the proof.
\end{proof}

\begin{proposition}\label{prop4}
Let $\gamma_i$ satisfy \eqref{sigma condition 1}, \eqref{sigma condition 2}, for $i=1,2$.

\begin{enumerate}

\item If $A\in\mathcal{A}$ and $\sigma_i(x) = A(\gamma_i (x))$, $x\in\overline\Omega$, for $i=1,2$, then we have
\begin{eqnarray}
& & ||E_{A(\gamma_i)}||_{C^{1, \beta}(\Omega)} \leq C_1,\qquad\textnormal{for}\:i=1,2,\label{Efield bound1}\\
& & ||E_{A(\gamma_1)} - E_{A(\gamma_2)}||_{L^2(\Omega)} \leq C_3||\gamma_1 - \gamma_2||_{L^2(\Omega)}, \label{Efield difference bound1} \\
& & ||E_{A(\gamma_1)} - E_{A(\gamma_2)}||_{H^1(\Omega)} \leq \tilde{C_3}||\gamma_1 - \gamma_2||_{H^1(\Omega)}, \label{Efield difference bound1 in H1}
\end{eqnarray}
where $E_{A(\gamma_i)}$ is the unique solution to \eqref{Maxwell}, with $\sigma_i = A(\gamma_i)$, for $i=1,2$ and $C_1$, $C_3$, $\tilde{C_3}$ are positive constants depending only on $\Lambda$, $r_0$ and $L$.

\item Similarly, if and $A\in\mathcal{A}'$ and $\sigma_i (x)= A(x, \gamma_i (x))$, $x\in\overline\Omega$, for $i=1,2$, then we have
\begin{eqnarray}
& & ||E_{A(\cdot, \gamma_i)}||_{C^{1, \beta}(\Omega)} \leq C_2,\qquad\textnormal{for}\:i=1,2,\label{Electric field bound2}\\
& & ||E_{A(\cdot, \gamma_1)} - E_{A(\cdot, \gamma_2)}||_{L^2(\Omega)} \leq C_4||\gamma_1 - \gamma_2||_{L^2(\Omega)}, \label{E field difference bound2}\\
& & ||E_{A(\cdot, \gamma_1)} - E_{A(\cdot, \gamma_2)}||_{H^1(\Omega)} \leq \tilde{C_4}||\gamma_1 - \gamma_2||_{H^1(\Omega)}, \label{E field difference bound2 in H1}
\end{eqnarray}
where $E_{A(\cdot, \gamma_i)}$ is the unique solution to \eqref{Maxwell}, with $\sigma_i = A(\cdot, \gamma_i)$, for $i=1,2$, with $C_2$, $C_4$, $\tilde{C_4}$ being positive constants depending only on $\Lambda$, $r_0$ and $L$.

\end{enumerate}

\end{proposition}

\begin{proof}
As the proof follows the same line of \cite[Proof of Proposition 2]{A-Q-S-Z}, we only highlight the main steps of case $(a)$ and only point out where case $(b)$ differs from it, as such modifications are minor. For case $(a)$ we start by noticing that for $\Tilde{E}$ = $\frac{1}{2}(-y, x, 0)$ we have that $\nabla \times (E_{A(\gamma_i)} - \Tilde{E}) = 0$, for $i=1,2$ and $E_{A(\gamma_i)}= \Tilde{E} + \nabla u$, where $u_i$ solves the Neumann problem
\begin{equation}\label{1.15}
      \Bigg\{\begin{array}{llll}
        \nabla \cdot (A(\gamma_i) \nabla u_i) & = -\nabla \cdot (A(\gamma_i)\Tilde{E}) \quad &\textnormal{in} \quad  \Omega;\\
        (A(\gamma_i) \nabla u_i + A(\gamma_i) \Tilde{E}) \cdot \nu & = 0 \quad & \textnormal{on} \quad \partial \Omega,\end{array}
        \end{equation}
for $i=1,2$. Observing that
\begin{equation}\label{E tilde L2}
||\Tilde{E}||_{L^2(\Omega)} \leq\frac{1}{2} M \big(\mu (\Omega)\big)^\frac{1}{2},
\end{equation}
where $\mu(\Omega)$ denotes the Lebesgue measure of $\Omega$, estimate \eqref{Efield bound1} is a straightforward consequence of
\eqref{sigma condition 2} and proposition \ref{proposition2.1}  (see also  \cite{A-Q-S-Z}).

To prove \eqref{Efield difference bound1}, \eqref{Efield difference bound1 in H1}, we start by simplifying our notation. We denote $E_{A(\gamma_i)}$, simply with $E_i$, for $i$ = 1, 2. Noting that $\nabla\times E_1 = \nabla\times E_2$, we write
\begin{equation}
    \nabla v = E_1 - E_2,
\end{equation}
where $v$ is the unique solution to

\begin{equation}\label{1.5}
      \Bigg\{\begin{array}{llll}
        \nabla \cdot (A(\gamma_1) \nabla v) & = -\nabla \cdot ((A(\gamma_1) - A(\gamma_2)) E_2) \quad &\textnormal{in} \quad  \Omega,\\
        (A(\gamma_1) \nabla v  + (A(\gamma_1) - A(\gamma_2)) E_2) \cdot \nu & = 0 \quad &\textnormal{on} \quad \partial \Omega\end{array}
        \end{equation}
and
\begin{equation}
\int_{\partial \Omega} v = 0.
\end{equation}
Thus by Cauchy Schwartz we obtain
\begin{equation}
    \int_\Omega A(\gamma_1) \nabla{v} \cdot \nabla{v}  dx \leq \int_\Omega ((A(\gamma_1) -A(\gamma_2)){E_2})\cdot\nabla{v} \: dx,
\end{equation}
which, combined with \eqref{Efield bound1} and the uniform ellipticity condition \eqref{ellipticity 1}, leads to
\begin{equation}\label{grad difference 1}
||\nabla{v}||_{L^2(\Omega)} \leq \Lambda C_1 ||A(\gamma_1) -A(\gamma_2)||_{L^2(\Omega)}.
\end{equation}
From the Lagrange Theorem, for every $x\in\Omega$, there exists $t(x)$, $0<t(x)<1$, such that
\begin{equation}
A(\gamma_1(x)) -A(\gamma_2(x)) = (\gamma_1 - \gamma_2)(x)\: \frac{d A(t)}{dt}\Big |_{t = c(x)},
\end{equation}
where $c(x) = \gamma_{2}(x)+t(x)\big(\gamma_{1}(x) - \gamma_{2}(x)\big)$. Hence we obtain
\begin{equation}
||\nabla v||_{L^2(\Omega)}\leq \Lambda C_1 ||\gamma_1 -\gamma_2||_{L^2(\Omega)} \Big|\Big|\frac{d A(t)}{dt}\Big |_{t = c(\cdot)} \Big|\Big|_{L^2(\Omega)}
\end{equation}
and
\begin{equation}
||E_1 - E_2||_{L^2(\Omega)}\leq C_3 ||\gamma_1 -\gamma_2||_{L^2(\Omega)},
\end{equation}
where $C_3 = \Lambda \mathcal{E}_1 C_1$. From the standard theory of elliptic equations
\begin{equation}
    ||v||_{H^2(\Omega)} \leq C||\nabla \cdot ((A(\gamma_1) - A(\gamma_2)) E_2)||_{L^2(\Omega)},
\end{equation}
where $C$ is a positive constant depending only on $\lambda$, $\Lambda$, $r_0$ and $L$, which implies \eqref{Efield difference bound1 in H1}.

For case $(b)$, where $A\in\mathcal{A'}$ and $\gamma_i\in C^{1,\beta}(\overline\Omega)$ is a scalar function satisfying \eqref{assumptions sigma}, for $i=1,2$, we have, by \eqref{E tilde L2} and proposition \ref{proposition3.3} that
\begin{equation}
||E_{A(\cdot,\gamma_i)}||_{L^2(\Omega)} \leq C_2,\qquad\textnormal{for}\: i=1,2,
\end{equation}
where  $C_2 = (1 + \Lambda \mathcal{F}_2)(\frac{1}{2} M (\mu (\Omega))^\frac{1}{2})$. With a similar argument to case $(a)$, from the Lagrange Theorem, for every $x\in\Omega$, there exists $t(x)$, $0<t(x)<1$, such that 
\begin{equation}
A(x, \gamma_1(x)) -A(x, \gamma_2(x)) = (\gamma_1 - \gamma_2)(x) \frac{\partial A(x,\:t)}{\partial t}\Big |_{t = c(x)},
\end{equation}
where $c(x) = \gamma_{2}(x)+t(x)\big(\gamma_{1}(x) - \gamma_{2}(x)\big)$, which leads to
\begin{equation}
||E_1 - E_2||_{L^2(\Omega)}\leq \Lambda C_1 ||\gamma_1 -\gamma_2||_{L^2(\Omega)} \Big|\Big|\frac{\partial A(\cdot,\: t)}{\partial t}\Big |_{t = c(\cdot)} \Big|\Big|_{L^2(\Omega)},
\end{equation}
hence
\begin{equation}
||E_1 - E_2||_{L^2(\Omega)}\leq C_4 ||\gamma_1 -\gamma_2||_{L^2(\Omega)},
\end{equation}
where $C_4 = \Lambda \mathcal{E}_2 C_2$.  The remaining \eqref{E field difference bound2 in H1} follows by a similar argument used to prove \eqref{Efield difference bound1 in H1} in case $(a)$.
\end{proof}


\subsection{Proof of the stability estimates}

\begin{proof}[Proof of Theorem \ref{theorem 1}]

We write the difference in the data $F_1 - F_2$ as
\begin{equation}
\begin{split}
 F_1 - F_2 &= \nabla \cdot \Big(A(\gamma_1) E_1 \times B_0\Big) - \nabla \cdot \Big(A(\gamma_2) E_2 \times B_0\Big) \\
&= \nabla \cdot \Big((A(\gamma_1) - A(\gamma_2))E_1 \times B_0\Big)  + \nabla \cdot \Big(A(\gamma_2) (E_1 - E_2) \times B_0\Big) 
\end{split}
\end{equation}
and from the Lagrange Theorem, for every $x\in\Omega$, there exists $t(x)$, $0<t(x)<1$, such that
\begin{equation}
F_1 - F_2 = \nabla \cdot \Big((\gamma_1 - \gamma_2) \frac{dA}{dt}(t)\Big|_{t = c(\cdot)}E_1 \times B_0\Big)  + \nabla \cdot \Big(A(\gamma_2)\:(E_1 - E_2) \times B_0\Big),
\end{equation}
where $c(x) = \gamma_{2}(x)+t(x)\big(\gamma_{1}(x) - \gamma_{2}(x)\big)$.
Then, we can split $F_1 - F_2$ into three contributions
\begin{equation}
F_1 - F_2 = I_1 + I_2 + I_3,
\end{equation}
where
\begin{equation}
\begin{split}
& I_1 = \nabla \cdot \Big((\gamma_1 - \gamma_2) E_1 \times B_0\Big),\\
&I_2 = \nabla \cdot \Big((\gamma_1 - \gamma_2) \Big(\frac{dA}{dt}(t)\Big|_{t = c(\cdot)} - I\Big)E_1 \times B_0\Big),\\
&I_3 = \nabla \cdot \Big(A(\gamma_2) (E_1 - E_2) \times B_0\Big),
\end{split}
\end{equation}
where $I$ denotes the $3\times 3$ identity matrix. To extract information about $\gamma_1 - \gamma_2$ in $\Omega$ we multiply the data difference $F_1 - F_2$ by $\gamma_1 - \gamma_2$ and integrate over $\Omega$.  We do this by considering each contribution $I_i$, i = 1, 2, 3, to the data separately. Starting with $I_1$, we obtain
\begin{equation}\label{I1}
\begin{split}
\int_\Omega (\gamma_1 - \gamma_2)I_1 dx
&= \int_\Omega (\gamma_1 - \gamma_2)^2 \nabla \cdot ( E_1 \times B_0) \\
&+ \frac{1}{2} \nabla(\gamma_1 - \gamma_2)^2 \cdot (E_1 \times B_0) dx \\
&= \frac{1}{2} \int_\Omega (\gamma_1 - \gamma_2)^2 \nabla \cdot (E_1 \times B_0) dx \\
&= \frac{1}{2} || \gamma_1 - \gamma_2||^2_{L^2(\Omega)},
 \end{split}
 \end{equation}
where in the second equality of \eqref{I1} we performed integration by parts and used the fact that $\gamma_1$ and $\gamma_2$ share the same boundary values. In the last equality, we used the fact that $\nabla \cdot (E_1 \times B_0) =1$. \\

Multiplying $I_2$ by $\gamma_1 - \gamma_2$, we have
 \begin{equation}\label{I2}
 \begin{split}
\left|\int_\Omega (\gamma_1 - \gamma_2)I_2 dx\right|   &= \Big|\int_\Omega (\gamma_1 - \gamma_2)^2 \nabla \cdot \left(\Big(\frac{dA}{dt}(t)\Big|_{t=c(\cdot)} - I\Big)E_1 \times B_0\right) \\
&+ (\gamma_1 - \gamma_2)\nabla(\gamma_1 - \gamma_2) \cdot \left(\Big(\frac{dA}{dt}(t)\Big|_{t=c(\cdot)} - I\Big)E_1 \times B_0\right) dx\Big| \\
&= \Big|\int_\Omega(\gamma_1 - \gamma_2)\nabla(\gamma_1 - \gamma_2) \cdot \left(\Big(\frac{dA}{dt}(t)\Big|_{t=c(\cdot)} - I\Big)E_1 \times B_0\right) dx\Big| \\
&\leq \mathcal{C} C_1|\mathcal{E}_1-1|||\gamma_1 - \gamma_2||^2_{L^2(\Omega)},
 \end{split}
 \end{equation}
where in the second equality of \eqref{I2} we again performed integration by parts and again recalled that $\gamma_1$ and $\gamma_2$ share the same boundary. In the inequality, estimate \eqref{Derivative assumption}, \eqref{Gradient sigma difference assumption} and \eqref{Efield bound1} have been used. \\

Finally, multiplyig $I_3$ by $\gamma_1 - \gamma_2$, leads to
 \begin{equation}\label{I3}
 \begin{split}
\left|\int_\Omega (\gamma_1 - \gamma_2)I_3 dx\right|
&= \left|\int_\Omega \nabla(\gamma_1 - \gamma_2) \cdot \Big(A(\gamma_2) (E_1 - E_2) \times B_0\Big) dx \right| \\
&\leq \mathcal{C} \Lambda C_3||\gamma_1 - \gamma_2||^2_{L^2(\Omega)},
 \end{split}
 \end{equation}
where the first equality in \eqref{I3} is derived again by performing integration by parts and noticing that $\gamma_1$ and $\gamma_2$ share the same boundary values. In the inequality, \eqref{ellipticity 1}, assumption \eqref{Gradient sigma difference assumption} and \eqref{Efield difference bound1} have been used. \\

Hence, for $ \mathcal{C} C_1(\mathcal{E}_1 +1) + \mathcal{C} \Lambda C_3 <\frac{1}{2}$, we obtain
\begin{equation}
    \int_{\Omega} (\gamma_1 - \gamma_2)(F_1 - F_2)dx \geq C ||\gamma_1 - \gamma_2||^2_{L^2(\Omega)},
\end{equation}
where $C = \frac{1}{2} - \mathcal{C} \Big(C_1(\mathcal{E}_1 +1) + \Lambda C_3\Big)>0$, which concludes the proof. 
\end{proof}


\begin{proof}[Proof of Theorem \ref{theorem 2}]
Similarly to the proof of theorem \ref{theorem 1}, we write the data difference as
\begin{equation}
\begin{split}
 \tilde{F}(\gamma_1) - \tilde{F}(\gamma_2) &= \nabla \cdot (A(\cdot, \gamma_1) E_1 \times B_0) - \nabla \cdot (A(\cdot, \gamma_2) E_2 \times B_0) \\
&= \nabla \cdot ((A(\cdot, \gamma_1) - A(\cdot, \gamma_2))E_1 \times B_0)  + \nabla \cdot (A(\cdot,\gamma_2) (E_1 - E_2) \times B_0).
\end{split}
\end{equation}
From the Lagrange Theorem, for every $x\in\Omega$, there exists $t(x)$, $0<t(x)<1$, such that
\begin{equation}
A(x, \gamma_1) -A(x, \gamma_2) = (\gamma_1 - \gamma_2)(x) \frac{\partial A(x, t)}{\partial t}\Big |_{t = c(x)},
\end{equation}
where $c(x) = \gamma_{2}(x)+t(x)\big(\gamma_{1}(x) - \gamma_{2}(x)\big)$. Therefore we obtain

\begin{equation}
\tilde{F}_1 - \tilde{F}_2 = \nabla \cdot ((\gamma_1 - \gamma_2) \frac{\partial A(\cdot, t)}{\partial t}\Big |_{t = c(x)}E_1 \times B_0)  + \nabla \cdot (A(\cdot,\:\gamma_2)(E_1 - E_2) \times B_0).
\end{equation}

By splitting again $\tilde{F}_1 - \tilde{F}_2$ as
\begin{equation}
\tilde{F}_1 - \tilde{F}_2 = I_1 + I_2 + I_3,
\end{equation}
where
\begin{equation}
\begin{split}
& I_1 = \nabla \cdot \Big((\gamma_1 - \gamma_2) E_1 \times B_0\Big),\\
&I_2 = \nabla \cdot \left((\gamma_1 - \gamma_2) \Big(\frac{\partial A(\cdot, t)}{\partial t}\Big|_{t = c(\cdot)} - I \Big)E_1 \times B_0\right),\\
&I_3 = \nabla \cdot \Big(A(\cdot,\:\gamma_2) (E_1 - E_2) \times B_0\Big),
\end{split}
\end{equation}
with a similar procedure as that of the proof of Theorem \ref{theorem 1} we obtain the following estimates
 \begin{equation}\label{I1 bis}
 \begin{split}
\int_\Omega (\gamma_1 - \gamma_2)I_1 dx
&= \frac{1}{2} || \gamma_1 - \gamma_2||^2_{L^2(\Omega)}.
 \end{split}
 \end{equation}
 \begin{equation}\label{I2 bis}
 \begin{split}
\left|\int_\Omega (\gamma_1 - \gamma_2)I_2 dx \right| &=\left|\int_\Omega(\gamma_1 - \gamma_2)\nabla(\gamma_1 - \gamma_2) \cdot \left(\Big(\frac{\partial A(\cdot,\: t)}{\partial t}\Big|_{t=c(\cdot)} - I\Big)E_1 \times B_0\right) dx\right|\\
 &\leq \mathcal{C} C_2|\mathcal{E}_2-1|||\gamma_1 - \gamma_2||^2_{L^2} .
 \end{split}
 \end{equation}
 \begin{equation}\label{I3 bis}
 \begin{split}
\left|\int_\Omega (\gamma_1 - \gamma_2)I_3 dx\right| &= \left|\int_\Omega \nabla(\gamma_1 - \gamma_2) \cdot (A(\cdot,\:\gamma_2(\cdot))\:(E_1 - E_2) \times B_0) dx \right| \\
&\leq \mathcal{C} \Lambda C_4 ||\gamma_1 - \gamma_2||^2_{L^2}
 \end{split}
 \end{equation}
By combining \eqref{I1 bis} - \eqref{I3 bis} together and choosing the parameters such that$(\mathcal{C} C_2(\mathcal{E}_2 +1) + \mathcal{C} \Lambda C_4)<\frac{1}{2}$, we finally obtain
\begin{equation}
    \int_{\Omega} (\gamma_1 - \gamma_2)(\Tilde{F}(\gamma_1) - \Tilde{F}(\gamma_2))dx \geq \Tilde{C} ||\gamma_1 - \gamma_2||^2_{L^2(\Omega)},
\end{equation}
where $\Tilde{C} = \frac{1}{2} - (\mathcal{C} C_2(\mathcal{E}_2 +1) + \mathcal{C} \Lambda C_4)>0$, which concludes the proof.
\end{proof}



\section{Reconstruction of $\gamma$}\label{reconstruction}


\subsection{The functional framework and the algorithm}\label{algorithm}
To ensure well-posedness of the inverse problem we introduce a functional framework, similar to that in \cite{A-Q-S-Z}, based assuming that $\gamma$ is known on the boundary of $\Omega$ and by letting the true scalar function to be  $\gamma_*$ (the unknown conductivity to be reconstructed) in $\Omega$, belonging to a bounded convex subset $\Tilde{S}$ of $C^{1, \beta}(\Omega)$, defined by
\begin{equation}
    \Tilde{S} = \{ \gamma := \gamma_0 + \alpha | \alpha \in S\},
\end{equation}
where $\gamma_0$ satisfies \eqref{sigma condition 1}, \eqref{sigma condition 2} and

\begin{equation}
    S = \{\alpha \in C^{1, \beta}_0(\Omega) | \lambda^{-1} \leq \alpha + \gamma_0 \leq \lambda, || \nabla \alpha||_{L^2(\Omega)} \leq \mathcal{C}||\alpha||_{L^2(\Omega)}, ||\alpha||_{L^2(\Omega)} \leq \tilde{k}\},
\end{equation}
for some positive constant $\tilde{k}$. \\
For $A\in\mathcal{A}$, the true conductivity is denoted by $A_* := A(\gamma_*)$  and the forward operator is given by
\begin{equation}
    F(\gamma) = \nabla \cdot (A (\gamma) E_{A(\gamma)} \times B_0),
\end{equation}
where $E_A := E_{A(\gamma)}$ solves
\begin{equation}
      \Bigg\{\begin{array}{llll}
         \nabla \times E_{A} & = B_0, \quad &\text{in} \quad \Omega,\\
        \nabla \cdot A(\gamma) E_{A} & = 0, \quad &\text{on} \quad  \Omega,\\        A(\gamma) E_{A} \cdot \nu & = 0, \quad &\text{on} \quad \partial \Omega
        \end{array} \end{equation}
 and the scalar function $\gamma$ is updated by solving a stationary advection-diffusion equation as shown in the algorithm below.\\
     

       \textbf{The Algorithm} (reconstruction)
\begin{itemize}
\item We choose an initial conductivity $\gamma_1 \in \Tilde{S}$ and set $k=1$.
\item At step $k$, defining $A_k:= A(\gamma_k)$, we solve the boundary value problem
\begin{equation} \label{Algorithm1}
      \Bigg\{\begin{array}{llll}
         \nabla \times E_k & = B_0, \quad &\text{in} \quad \Omega,\\
        \nabla \cdot A_k E_k & = 0, \quad &\text{on} \quad  \Omega,\\
        A_k E_k \cdot \nu & = 0, \quad &\text{on} \quad \partial \Omega.
       \end{array} \end{equation}
        to update the electric field $E_k$. 
\item By solving the stationary advection-diffusion equation with the inflow boundary condition, we calculate the updated conductivity $A_{k+1/2}:= A(\gamma_{k+1})$ as follows
\begin{equation}
      \Big\{\begin{array}{llll}\label{Step 3}
         \nabla \cdot (A_{k+1/2}E_k \times B_0) &= \nabla \cdot (A_*E_{A_*} \times B_0), &\text{in} \quad \Omega,\\
        \gamma_{k+1/2} & = \gamma_*, &\text{on} \quad \partial \Omega^-,
        \end{array} \end{equation}
        where
       \begin{equation}\label{inflow}
        \partial \Omega^- = \left\{x \in \partial \Omega | A(\gamma_k(x)) E_k(x) \times B_0 \cdot \nu (x) <0\right\}.
        \end{equation}
       
    \item We then let $\gamma_{k+1} = T[\gamma_{k+1/2}]$ where $T$ is the Hilbert projection operator onto $\Tilde{S}$, where the projection $T$ is needed as $\gamma_k$ might not be in $\Tilde{S}$. Finally, we set $k = k + 1$ and go to \eqref{Algorithm1}.
\end{itemize}

For $A\in\mathcal{A'}$ the true conductivity is denoted by $A_* := A(\cdot, \gamma_*)$  and the forward operator is given by
\begin{equation}
    \Tilde{F}(\gamma) = \nabla \cdot (A (\cdot, \gamma) E_{A(\cdot, \gamma)} \times B_0),
\end{equation}
where $E_A := E_{A(\cdot, \gamma)}$ solves
\begin{equation}
      \Bigg\{\begin{array}{llll}
         \nabla \times E_{A} & = B_0, \quad &\text{in} \quad \Omega,\\
        \nabla \cdot A(\cdot, \gamma) E_{A} & = 0, \quad &\text{on} \quad  \Omega,\\        A(\cdot, \gamma) E_{A} \cdot \nu & = 0, \quad &\text{on} \quad \partial \Omega
        \end{array} \end{equation}
and the scalar function $\gamma$ together with the electric field $E$ are updated as per Algorithm 1 above, where the inflow part of the boundary is in this case defined by
  \begin{equation}\label{inflow 2}
        \partial \Omega^- = \left\{x \in \partial \Omega | A(x, \gamma_k(x)) E_k(x) \times B_0 \cdot \nu (x) <0\right\}.
        \end{equation}


\subsection{Convergence analysis}\label{convergence analysis}
Under a smallness condition on $\frac{dA(t)}{dt}$ and $\frac{\partial A(\cdot,t)}{dt}$, on $[\lambda^{-1}, \lambda]$ with respect to $I$, for $A\in\mathcal{A}$ and $A\in\mathcal{A'}$, respectively (see \eqref{Double derivative bound1} and \eqref{Double derivative bound2}, respectively), we obtain the following convergence results.

\begin{theorem}\label{theorem CA1}
Suppose that the true conductivity $\gamma_* \in \Tilde{S}$. Let $A\in\mathcal{A}$ and, additionally, assume that
\begin{equation}\label{derivative A 1}
\frac{dA}{dt}\in C^{1}\left([\lambda^{-1}, \lambda], Sym\right)
\end{equation}
and that there is a constant $\Tilde{\mathcal{E}}_1$, $0<\Tilde{\mathcal{E}}_1<1$ such that
\begin{equation}\label{Double derivative bound1}
\sup_{t\in [\lambda^{-1}, \lambda]} \sup_{x\in\Omega}\left|\nabla\cdot \left[\left(\frac{dA}{dt} - I\right) E_k\times B_0\right]\right|\leq \Tilde{\mathcal{E}}_1.
\end{equation}
Assume that either
\begin{enumerate}
\item \[\gamma_k = \gamma_* , \qquad\textnormal{on}\quad\partial\Omega ,\quad\textnormal{for\:any}\:k,\quad k\geq 1 ;\]
or
\item \begin{eqnarray*}
& &\gamma_k = \gamma_* ,\qquad\textnormal{on}\quad\partial\Omega^{-}\quad\textnormal{and}\quad \sup_{t\in[\lambda^{-1}, \lambda]}\frac{dA}{dt} E_k\times B_0\cdot\nu\geq 0,\qquad\textnormal{on}\quad \partial\Omega^{+}, \\
& &\textnormal{for\:any}\:k,\quad k\geq 1.
\end{eqnarray*}
\end{enumerate}
Then, for an appropriate choice of  $\mathcal{C}$, $\Lambda$, $C_1$, $\tilde{C_3}$, $\mathcal{E}_1$ and $\tilde{\mathcal{E}_1}$, there exists a constant $C$, $0<C < 1$ such that
\begin{equation}
    || \gamma_{k+1} - \gamma_*||_{L^2(\Omega)} \leq C ||\gamma_k - \gamma_*||_{L^2(\Omega)}. \\ \\
\end{equation}
\end{theorem}

\begin{theorem}\label{theorem CA2}
Suppose that the true conductivity $\gamma_* \in \Tilde{S}$. Let $A\in\mathcal{A'}$, and, additionally, assume that 
\begin{equation}
\frac{\partial A (\cdot, t)}{\partial t} \in C^1(\overline\Omega\times[\lambda^{-1}, \lambda], Sym)
\end{equation}
and there exists a constant $\Tilde{\mathcal{E}_2}$, $0< \Tilde{\mathcal{E}_2} < 1$ such that
\begin{equation}\label{Double derivative bound2}
    \underset{t \in [\lambda^{-1}, \lambda]}{\sup} \sup_{x\in\Omega}\quad \Big|\nabla \cdot \Big[ \Big(\frac{\partial A}{\partial t}(\cdot, t)\Big|_{t = c(\cdot)} - I \Big) E_k \times B_0 \Big] \Big| \leq \tilde{\mathcal{E}_2}. 
\end{equation}
Assume that either
 \begin{enumerate}
\item \[\gamma_k = \gamma_* , \qquad\textnormal{on}\quad\partial\Omega , \quad\textnormal{for\:any}\:k,\quad k\geq 1 ;\]
or
\item 
\begin{eqnarray*}
& & \gamma_k = \gamma_* ,\qquad\textnormal{on}\quad\partial\Omega^{-}\quad\textnormal{and}\quad \sup_{t\in[\lambda^{-1}, \lambda]}\frac{\partial A (\cdot, t)}{\partial t} E_k\times B_0\cdot\nu\geq 0,\qquad\textnormal{on}\quad \partial\Omega^{+},\\
& &\textnormal{for\:any}\:k,\quad k\geq 1.
\end{eqnarray*}
\end{enumerate}
Then, for an appropriate choice of $\mathcal{C}$, $\Lambda$, $C_2$, $\tilde{C_4}$ $\mathcal{E}_2$and $\tilde{\mathcal{E}_2}$, there exists a constant $C$ $0< C < 1$ such that
\begin{equation}
    || \gamma_{k+1} - \gamma_*||_{L^2(\Omega)} \leq C ||\gamma_k - \gamma_*||_{L^2(\Omega)}. \\ \\
\end{equation}
\end{theorem}

\begin{proof}[Proof of Theorem \ref{theorem CA1}]
As $T$ projects into $\Tilde{S}$, which is convex, we have that  $|| \gamma_{k+1} - \gamma_*||_{L^2(\Omega)} \leq ||\gamma_{k+1/2} - \gamma_*||_{L^2(\Omega)}$. It is left to estimate $||\gamma_{k+1/2} - \gamma_*||_{L^2(\Omega)}$. \\

Subtracting $\nabla \cdot (A_* E_k \times B_0)$ from both sides of \eqref{Step 3} we get

\begin{equation}
    \nabla \cdot ((A_{k+1/2}-A_*) E_k \times B_0) = \nabla \cdot (A_* (E_* - E_k) \times B_0).
\end{equation}
Multiplying by $\gamma_{k+1/2}-\gamma_*$ and integrating over $\Omega$ yields

\begin{equation}
    \int_{\Omega} (\gamma_{k+1/2}-\gamma_*)\nabla \cdot ((A_{k+1/2}-A_*) E_k \times B_0) = \int_{\Omega} (\gamma_{k+1/2}-\gamma_*) \nabla \cdot (A_* (E_* - E_k) \times B_0).
\end{equation}
Using the property that 
\begin{equation}\label{A C1}
A\in C^1\left([\lambda^{-1}, \lambda],Sym\right),
\end{equation}
from the Lagrange Theorem, for every $x\in\Omega$, there exists $t(x)$, $0<t(x)<1$, such that
\begin{equation}
A(\gamma_{k+1/2}) -A(\gamma_*) = (\gamma_{k+1/2} - \gamma_*) \frac{d A}{dt}(t)\Big|_{t = c(x)},
\end{equation}
where $c(x) = \gamma_{k+1/2}(x)+t(x)\big(\gamma_*(x) - \gamma_{k+1/2}(x)\big)$, hence we get

\begin{equation}\label{eq 1}
\int_{\Omega} (\gamma_{k+1/2}-\gamma_*)\nabla \cdot ((\gamma_{k+1/2} - \gamma_*) \frac{d A}{dt}(t)\Big|_{t = c(\cdot)} E_k \times B_0)
= \int_{\Omega} (\gamma_{k+1/2}-\gamma_*) \nabla \cdot (A_* (E_* - E_k) \times B_0).
\end{equation}

The term on the left hand side of \eqref{eq 1} can be estimated from below as follows
\begin{equation}\label{eq 0}
    \begin{split}
&\int_{\Omega} (\gamma_{k+1/2}-\gamma_*)\nabla \cdot ((\gamma_{k+1/2} - \gamma_*) \frac{d A}{dt}(t)\Big|_{t = c(\cdot)} E_k \times B_0)dx \\
=&  \frac{1}{2}\int_{\partial\Omega}(\gamma_{k+1/2} - \gamma_*)^2 \cdot \Big(\frac{dA}{dt}(t)\Big|_{t = c(\cdot)}E_k \times B_0\Big)\cdot \nu ds(x) \\
&+ \frac{1}{2}\int_\Omega(\gamma_{k+1/2} - \gamma_*)^2\nabla \cdot \Big(\frac{d A}{dt}(t)\Big|_{t = c(\cdot)}E_k \times B_0\Big) dx,\\
\geq & \frac{1}{2}\int_{\partial\Omega}(\gamma_{k+1/2} - \gamma_*)^2 \cdot \Big(\frac{dA}{dt}(t)\Big|_{t = c(\cdot)}E_k \times B_0\Big)\cdot \nu ds(x) \\
&+ \frac{1}{2}||(\gamma_{k+1/2} - \gamma_*)||^2_{L^2(\Omega)}\\
&- \frac{1}{2}\int_\Omega|\gamma_{k+1/2} - \gamma_*|^2 \quad \underset{\Omega}{\sup} \Big| \nabla \cdot \Big(\Big(\frac{d A}{dt}(t)\Big|_{t = c(\cdot)} - I \Big)E_k \times B_0\Big)\Big| dx.
       \end{split}
 \end{equation} 
In the second equality in \eqref{eq 0} we performed integration by parts twice. In either cases $1.$ or $2.$, \eqref{eq 0}, combined with \eqref{Double derivative bound1}, leads to
\begin{equation}
\begin{split}
&\int_{\Omega} (\gamma_{k+1/2}-\gamma_*)\nabla \cdot ((\gamma_{k+1/2} - \gamma_*) \frac{d A}{dt}(t)\Big|_{t = c(\cdot)} E_k \times B_0)dx \\
& \geq  \frac{1}{2}(1-\tilde{\mathcal{E}_1}) ||(\gamma_{k+1/2} - \gamma_*)||^2_{L^2(\Omega)}. 
\end{split}
\end{equation}
The term on the right hand side of \eqref{eq 1} can be estimated from above as 
\begin{equation}
    \begin{split}
        &\Big|\int_{\Omega}
        (\gamma_{k+1/2}-\gamma_*) \nabla \cdot (A_* (E_* - E_k) \times B_0)\Big| \\
        \leq& ||\gamma_{k+1/2} - \gamma_*||_{L^2(\Omega)}||A_* (E_* - E_k) \times B_0||_{H^1(\Omega)}, \\
        \leq& (\tilde{C_3}+\mathcal{C})(\Lambda +\mathcal{E}_1 K)||\gamma_{k+1/2} - \gamma_*||_{L^2(\Omega)}||\gamma_* - \gamma_k||_{L^2(\Omega)},
    \end{split}
\end{equation}
where we combined estimates \eqref{Derivative assumption}, \eqref{sigma condition 2}, together with \eqref{ellipticity 1}, \eqref{Gradient sigma difference assumption} and \eqref{Efield difference bound1 in H1}. \\
Choosing $\frac{(\tilde{C_3}+ \mathcal{C})(\Lambda +\mathcal{E}_1 K)}{1- \tilde{\mathcal{E}_1}} \leq \frac{1}{2}$, we finally derive
 \begin{equation}
     ||\gamma_{k+1/2} - \gamma_*||_{L^2(\Omega)} \leq C||\gamma_{k} - \gamma_*||_{L^2(\Omega)},
 \end{equation}
 where $0<C<1$.
\end{proof}

\begin{proof}[Proof of Theorem \ref{theorem CA2}]
As this proof is very similar to the one above, we only point where the two proofs slightly differ. In particular, we take care to point out what are the \textit{a-priori} constants that come into play. As above, the goal is to estimate $||\gamma_{k+1/2} - \gamma_*||_{L^2(\Omega)}$. Subtracting $\nabla \cdot (A_* E_k \times B_0)$ from both sides of \eqref{Step 3} leads to
\begin{equation}
    \nabla \cdot ((A_{k+1/2}-A_*) E_k \times B_0) = \nabla \cdot (A_* (E_* - E_k) \times B_0).
\end{equation}
Multiplying by $\gamma_{k+1/2}-\gamma_*$ and integrating over $\Omega$ yields
\begin{equation}
    \int_{\Omega} (\gamma_{k+1/2}-\gamma_*)\nabla \cdot ((A_{k+1/2}-A_*) E_k \times B_0) = \int_{\Omega} (\gamma_{k+1/2}-\gamma_*) \nabla \cdot (A_* (E_* - E_k) \times B_0).
\end{equation}
Using the property that 
\begin{equation}\label{A C1 bis}
A\in C^1\left(\overline\Omega\times[\lambda^{-1}, \lambda],Sym\right),
\end{equation}
from the Lagrange Theorem, for every $x\in\Omega$, there exists $t(x)$, $0<t(x)<1$, such that
\begin{equation}
A(x, \gamma_{k+1/2}) -A(x, \gamma_*) = (\gamma_{k+1/2} - \gamma_*)(x) \frac{\partial A(x, t)}{\partial t}\Big |_{t = c(x)},
\end{equation}

where $c(x) = \gamma_{k+1/2}(x)+t(x)\big(\gamma_*(x) - \gamma_{k+1/2}(x)\big)$, hence we have

\begin{equation}\label{eq 1 bis}
\int_{\Omega} (\gamma_{k+1/2}-\gamma_*)\nabla \cdot ((\gamma_{k+1/2} - \gamma_*) \frac{\partial A(\cdot, t)}{\partial t}\Big |_{t = c(\cdot)} E_k \times B_0) = \int_{\Omega} (\gamma_{k+1/2}-\gamma_*) \nabla \cdot (A_* (E_* - E_k) \times B_0).
\end{equation}

Arguing as above, the left hand side of \eqref{eq 1 bis} can be estimated from below as 
\begin{equation}\label{eq 0 bis}
    \begin{split}
&\int_{\Omega} (\gamma_{k+1/2}-\gamma_*)\nabla \cdot ((\gamma_{k+1/2} - \gamma_*) \frac{\partial A(\cdot, t)}{\partial t}\Big |_{t = c(\cdot)} E_k \times B_0)dx \\
& \geq\frac{1}{2}\int_{\partial\Omega}(\gamma_{k+1/2} - \gamma_*)^2 \cdot \Big(\frac{\partial A(\cdot, t)}{\partial t}\Big |_{t = c(\cdot)}E_k \times B_0\Big)\cdot \nu ds(x) \\
&+ \frac{1}{2}||(\gamma_{k+1/2} - \gamma_*)||^2_{L^2(\Omega)}\\
&- \frac{1}{2}\int_\Omega|\gamma_{k+1/2} - \gamma_*|^2 \quad \underset{\Omega}{\sup} \Big| \nabla \cdot \Big(\Big(\frac{\partial A(\cdot, t)}{\partial t}\Big |_{t = c(\cdot)} - I \Big)E_k \times B_0\Big)\Big| dx.
 \end{split}
 \end{equation}
and again, in either case, $\textit{1.}$ or $\textit{2.}$, \eqref{eq 0 bis}, combined together with \eqref{Double derivative bound2}, leads to
 \begin{equation}
    \begin{split}
&\int_{\Omega} (\gamma_{k+1/2}-\gamma_*)\nabla \cdot ((\gamma_{k+1/2} - \gamma_*) \frac{\partial A(\cdot, t)}{\partial t}\Big |_{t = c(\cdot)} E_k \times B_0)dx \\
    &  \geq  \frac{1}{2}(1- \tilde{\mathcal{E}_2}) || (\gamma_{k+1/2} - \gamma_*)||^2_{L^2(\Omega)}.  \\
 \end{split}
 \end{equation}

The right hand side of \eqref{eq 1 bis} can be estimated from above as
\begin{equation}
    \begin{split}
        &\Big|\int_{\Omega}
        (\gamma_{k+1/2}-\gamma_*) \nabla \cdot (A_* (E_* - E_k) \times B_0)\Big| \\
         &     \leq (\tilde{C_4}+\mathcal{C})(\Lambda +\mathcal{E}_2(1+ K))||\gamma_{k+1/2} - \gamma_*||_{L^2(\Omega)}||\gamma_* - \gamma_k||_{L^2(\Omega)},
 \end{split}
 \end{equation}
by combining estimates \eqref{sigma condition 2} and \eqref{partial derivative assumption}, and the use of \eqref{ellipticity 2}, \eqref{Gradient sigma difference assumption} and \eqref{E field difference bound2 in H1}. \\ \\
Therefore, choosing $\frac{(\tilde{C_4}+\mathcal{C})(\Lambda +\mathcal{E}_2(1+ K))}{1- \tilde{\mathcal{E}_2}} \leq \frac{1}{2}$, we derive

 \begin{equation}
     ||\gamma_{k+1/2} - \gamma_*||_{L^2(\Omega)} \leq C ||\gamma_{k} - \gamma_*||_{L^2(\Omega)}.
 \end{equation}
\end{proof}


\section{Numerical Experiments}\label{numerics}

In this section we present several numerical experiments where we implement the algorithm described in section \ref{reconstruction}. In each case, a scalar reference medium $\gamma_* \in \tilde{S}$ (the true scalar function to be reconstructed) is chosen, along with a specific choice of a $3 \times 3$ matrix function $A$ belonging either to $\mathcal{A}$ or $\mathcal{A'}$. We start by considering a number of examples in which $A\in\mathcal{A}$ only (Examples 1 - 4). The more general case of $A\in\mathcal{A'}$ is exploited in Examples 5 and 6.\\

We then compute the synthetic acoustic source data
\begin{equation}\label{data} 
F(\gamma_*) = \nabla \cdot (A_{\gamma_*} E_* \times B_{0}), 
\end{equation}
where $E_*$ is the electric field solution to (\ref{Algorithm1}) corresponding to $\gamma_*$. For simplicity, in our examples we always choose reference conductivity $\gamma_{0} = 1$. We note that in the particular examples here, we found that the projection step of the algorithm (Step 4)  was not necessary for convergence of the iterates, and in two of the examples below, $\gamma_*$ is not in $C^1$ (Examples 4 and 6, where $\gamma_*$ is piecewise affine and piecewise constant, respectively). This is to test the reconstruction from non smooth data. Again, it will be part of future work to extend our theoretical framework of sections \ref{main assumptions} - \ref{reconstruction} to conductivities with lower regularity assumptions and a more general anisotropic structure.

To compute solutions to both  (\ref{Algorithm1}) and (\ref{Step 3}), we use the Python suite FEniCS.  As described in the proof of Proposition \ref{prop4}, we replace (\ref{Algorithm1}) with the Neumann problem (\ref{1.15}) and solve numerically for $u_{k+1}=u$ using a standard variational formulation and piecewise linear finite elements.  Recall then that $E_{k+1}$ is given by $$E_{k + 1} = \nabla u_{k + 1} + \tilde{E},$$ where $\tilde{E} = 0.5[-y, x, 0]^{T}$.  We note that while (\ref{1.15}) is always a linear conductivity equation, (\ref{Step 3}) changes more drastically for different choices of $A$.  To compute solutions to this generally nonlinear transport equation, we used two different approaches; a discontinuous Galekin method and a built-in FEniCS nonlinear solver.  We give more details about these in the first two examples.


\begin{figure}[htbp]
\centering
\begin{minipage}{.5\textwidth}
\centering
  \includegraphics[scale = 0.55]{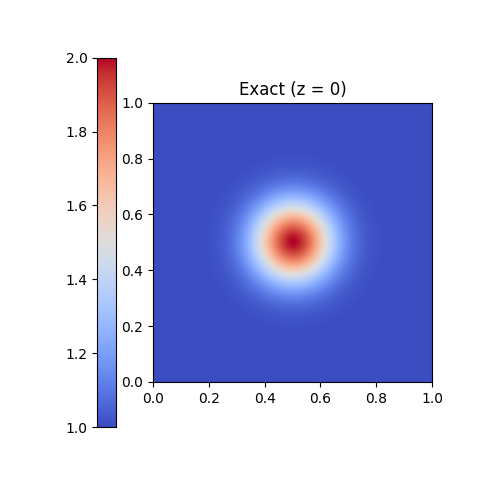}
  \caption{Reference conductivity $\gamma_*$ at $z = 0$}
\end{minipage}%
\begin{minipage}{.5\textwidth}
  \centering
  \includegraphics[scale = 0.55]{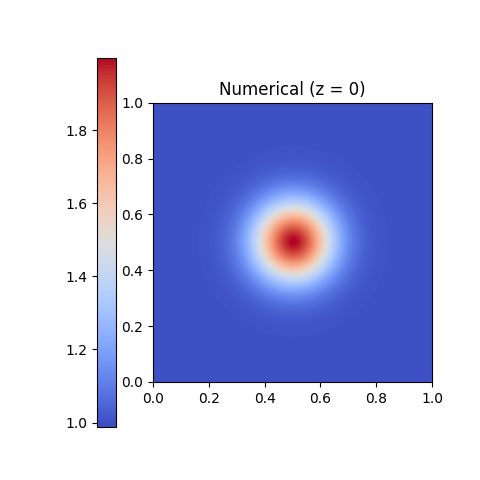}
  \caption{Reconstructed conductivity at $z = 0$}
\end{minipage}
\caption{ $A$ given by (\ref{D1}). The resulting transport equation is linear.  Ran for 10 iterations on a mesh of size 150 x 150 x 1.}
\label{figexample1}
\end{figure}

{\bf Example 1.} For our first example we choose $A$ to be
\begin{equation}\label{D1}
  A(\gamma)\ \ =\ \ \begin{bmatrix}
  \gamma & 0 & 0\\
  0 & \gamma & 0\\
  0 & 0 & 1
  \end{bmatrix}
  \end{equation}
and reference medium $\gamma_*$ to be the Gaussian curve independent of $z$:  $$
  \gamma_* = \exp{\left( -\frac{(x - 0.5)^{2}}{0.02} - \frac{(y - 0.5)^{2}}{0.02} \right)} + 1.$$  In this case the transport equation \eqref{Step 3} for $\gamma_{k + 1}$ becomes
\begin{equation*}
  \begin{cases}
    \nabla \cdot (\gamma_{k + 1}(\EB{k})) = \nabla \cdot (\gamma_*(E_* \times B_{0})), & \text{in } \Omega\\
    \gamma_{k + 1} = \gamma_{*} ,\hfill & \text{ on } \partial \Omega^{-},
  \end{cases}
\end{equation*}
which is 2D linear isotropic, as was studied previously in \cite{A-Q-S-Z}. To solve this linear transport equation, we use a discontinuous Galerkin (DG) method with upwinding, which we now describe briefly. Given a mesh $T_{h}$ of $\Omega$, we let $V$ be a finite dimensional space containing functions not necessarily continuous across mesh elements. We define the `\textit{upwinding}' term on the boundary of a triangle by  $$\bar{E}_{k} = 0.5((E_{k} \times B_{0}) \cdot \vec{n} + |(E_{k} \times B_{0}) \cdot \vec{n}|)$$  and denote the jump of a function on an internal edge $E$ by $$[[f]] = f^{+} - f^{-},$$ where $f^{+}$ ($f^{-}$) is its restriction to the edge from the outflow (inflow)  direction, respectively. Then the variational formulation for the numerical transport problem is: \\

find $\gamma_{k+1}= \gamma \in V$ such that, for all $v \in V$,
\begin{equation*}
  \int_{\partial \Omega}v\gamma (E_{k} \times B_{0})\cdot \vec{n}\, ds + \sum_{E}\int_{E} [[v]][[\gamma \bar{E}_{k}]]\, dS - \int_{\Omega} \nabla v \cdot \gamma (E_{k} \times B_{0})\, dx = \int_{\Omega} F(\gamma^*) v\, dx,
\end{equation*}
where $F(\gamma^*)$ is the source data (\ref{data}).
Note that the upwinding term provides numerical stability. We refer to \cite{B-S} for more details on DG methods and upwinding. Figure \ref{figexample1} shows the true $\gamma_*$, plotted alongside its numerical reconstruction using the above algorithm.

\begin{figure}[htbp]
 \centering
  \begin{minipage}{.5\textwidth}
   \centering
   \includegraphics[scale = 0.55]{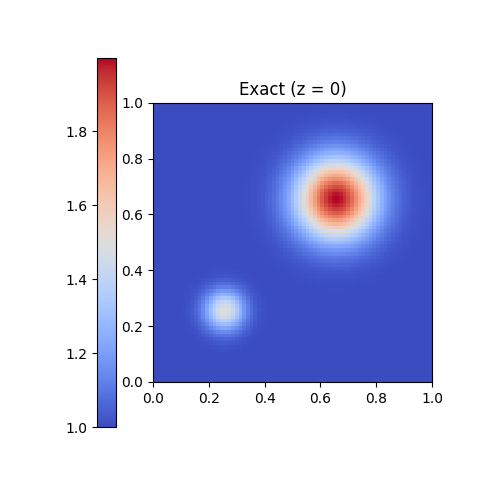}
    \caption{Reference conductivity $\gamma_*$ at $z = 0$}
  \end{minipage}%
  \begin{minipage}{.5\textwidth}
    \centering
   \includegraphics[scale = 0.55]{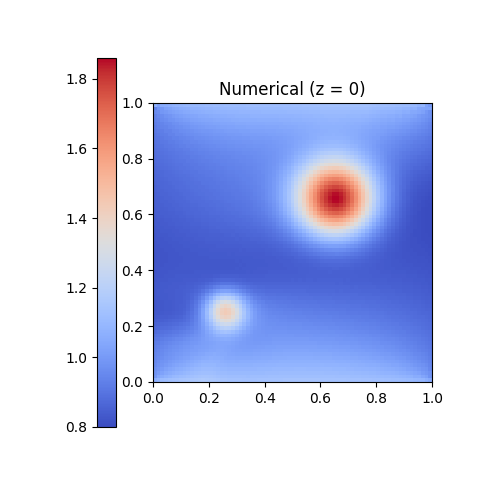}
    \caption{Reconstructed conductivity  at $z = 0$}
  \end{minipage}
  \caption{ $A$ given by (\ref{D2}), resulting in transport equation (\ref{transport2}). Ran for 10 iterations on a mesh of size 75 x 75 x 10.}
\label{figexample2}
\end{figure}

\vspace{0.5cm}

{\bf Example 2.} In this next example we add nonlinearity in the dependence of A on $\gamma$. To solve the transport equation \eqref{Step 3} in this and all of the examples that follow, instead of DG, we use the Fenics nonlinear solver with higher order Lagrange elements. We choose here
\begin{equation} \label{D2}
  A(\gamma)\ \ =\ \ \begin{bmatrix}
  0.4(\gamma + 1)^{2} & 0.01 & 0\\
  0.01 & 3\gamma & 0\\
  0 & 0 & \gamma
  \end{bmatrix}
\end{equation}
and reference conductivity $\gamma_*$ to be a sum of Gaussian curves given by
$$
  \gamma_* = \exp{\left( -\frac{(x - 0.65)^{2}}{0.02} - \frac{(y - 0.65)^{2}}{0.02} \right)} + 0.5 * \exp{\left( -\frac{(x - 0.25)^{2}}{0.05} - \frac{(y - 0.25)^{2}}{0.05} \right)}.
$$
With the definitions above, the transport problem \eqref{Step 3} becomes the nonlinear equation for $\gamma = \gamma_{k + 1}$:
\begin{equation}\label{transport2}
      a_{1}\gamma^{2} + a_{2}\gamma + a_{3}\gamma\gamma_{x} +a_{4}\gamma_{x} - a_{5}\gamma_{y} + c = F(\gamma_{*}),\qquad \text{ in } \quad\Omega,
      \end{equation}
with inflow condition $\gamma = \gamma_{*}$ on $\partial \Omega^{-}$,
where the spatially varying coefficients are functions of $E = [E_{1}, E_{2}, E_{3}]^{T}$, given by
 $a_{1} = 0.4E_{2, x}$,  $a_{2} = 0.8E_{2, x} - 3E_{1, y}$, $a_{3}= a_4 = 0.8E_{2} $,  $a_{5}= 3E_{1}$ and $c = 0.4E_{2, x}$. We show the reconstruction in Figure \ref{figexample2}.

\begin{figure}[htbp]
 \centering
  \begin{minipage}{.5\textwidth}
    \centering
    \includegraphics[scale = 0.55]{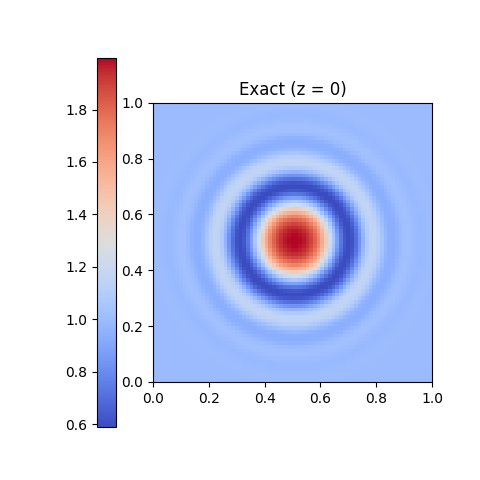}
    \caption{Reference conductivity $\gamma_*$ at $z = 0$}
 \end{minipage}%
  \begin{minipage}{.5\textwidth}
    \centering
    \includegraphics[scale = 0.55]{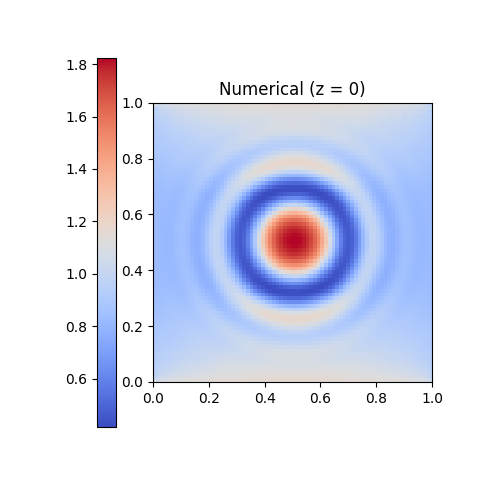}
   \caption{Reconstructed conductivity at $z = 0$}
  \end{minipage}
  \caption{A given by (\ref{D3}), yielding transport equation (\ref{transport3}). Ran for 10 iterations on a mesh of size 75 x 75 x 10.}
  \label{figexample3}
\end{figure}

\vspace{0.5cm}

{\bf Example 3.} In this next example we add nonlinearity on the off diagonals. Let
\begin{equation} \label{D3}
  A(\gamma)\ \ =\ \ \begin{bmatrix}
  0.4(\gamma + 1)^{2} & 0.01\gamma  (1 - \gamma) & 0\\
  0.01\gamma  (1 - \gamma) & 3\gamma & 0\\
  0 & 0 & \gamma
  \end{bmatrix}
\end{equation}
and let $\gamma_*$ be given by
$$
  \gamma_*(x, y) = \cos{(75(x - 0.5)^{2} + 75(y - 0.5)^{2})}\exp{\left( -\frac{(x - 0.5)^{2}}{2} - \frac{(y - 0.5)^{2}}{2} \right)} + 1.
$$
In this case the transport equation \eqref{Step 3} for $\gamma_{k+1}=\gamma$ becomes:
\begin{equation}\label{transport3}
    a_{1}\gamma^{2} + a_{2}\gamma\gamma_{y} + a_{3}\gamma_{y} + a_{4}\gamma +a_{5}\gamma_{x} + a_{6}\gamma\gamma_{x} + c = F(\gamma_{*}),\qquad  \text{in}\quad\Omega ,\end{equation}
  with    $\gamma = \gamma_{*}$ on $ \partial \Omega^{-}$, where the coefficients are given by
  $a_{1} = 0.4E_{2, x} + 0.01E_{1, x} - 0.01E_{2, y}$,  $a_{2} = 0.898E_{2}$,
   $a_{3} = 0.01E_{2} - 3E_{1}$,
    $a_{4}= 0.8E_{2, x} - 0.01E_{1, x} + 0.01E_{2, y} - 3E_{1, y}$,
  $a_{5} = 0.8E_{2} - 0.01E_{1}$,
 $a_{6} = 0.02E_{1}$, and
 $c= 0.4E_{2, x}$.
Reconstruction results are in Figure \ref{figexample3}.

\begin{figure}[htbp]
\centering
  \begin{minipage}{.5\textwidth}
   \centering
    \includegraphics[scale = 0.55]{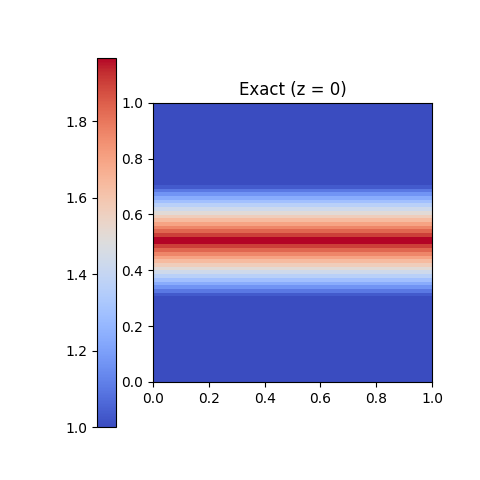}
    \caption{Reference conductivity $\gamma_*$ at $z = 0$}
  \end{minipage}%
  \begin{minipage}{.5\textwidth}
    \centering
    \includegraphics[scale = 0.55]{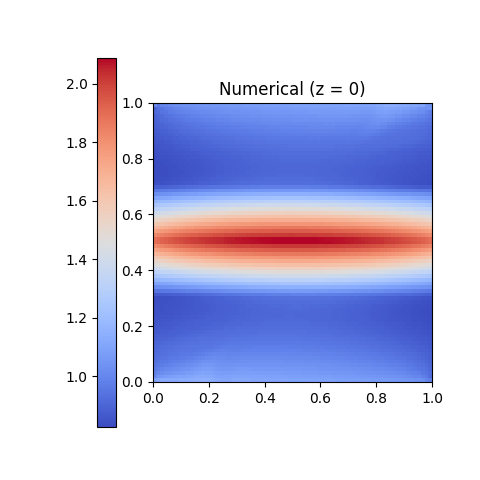}
    \caption{Reconstructed conductivity at $z = 0$}
 \end{minipage}
  \caption{A given by (\ref{D4}). Ran for 10 iterations on a mesh of size 75 x 75 x 10.}
\label{figexample4}
\end{figure}

\vspace{0.5cm}

{\bf Example 4.}
Here we choose a piecewise affine $\gamma_*$ to test the reconstruction of non smooth conductivity. Let
\begin{equation}\label{D4}
  A(\gamma)\ \ =\ \ \begin{bmatrix}
  0.4(\gamma + 1)^{2} & \frac{1}{\gamma + 20} & 0\\
  \frac{1}{\gamma + 20} & 3\gamma & 0\\
  0 & 0 & \gamma
  \end{bmatrix}
\end{equation}
and define  $\gamma_*$ by
$$
  \gamma_*(x, y) = \begin{cases}
  1 + 5(x - 0.3) & ,\ \ \ 0.3 \leq x \leq 0.5\\
  2 - 5(x - 0.5) & ,\ \ \ 0.5 \leq x \leq 0.7\\
  1 & ,\ \ \ \text{else}
  \end{cases}
$$
With the definitions above, the transport problem (\ref{Step 3}) becomes: find $\gamma = \gamma_{k + 1}$ such that
\begin{multline}
    (E_{2,x} - E_{1,y})\gamma
    + (E_{2})\gamma_{x}
    - \left( \frac{E_{2}}{(\gamma + 20)^{2}} + E_{1} \right)\gamma_{y} \\
    +\left( \frac{1}{\gamma + 20} \right)\left( E_{2,y} - E_{1,x} + \frac{E_{1,x}}{\gamma + 20} \right) = F(\gamma^*),\qquad\text{in}\quad\Omega ,
    \end{multline}
 with    $\gamma = \gamma_{*}$ on $ \partial \Omega^{-}$. Numerical reconstructions are show in Figure \ref{figexample4}.

\begin{figure}[htbp]
 \centering
 \begin{minipage}{.5\textwidth}
    \centering
   \includegraphics[scale = 0.55]{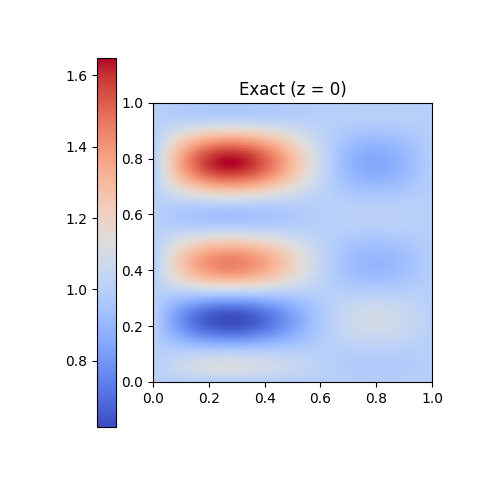}
    \caption{Reference conductivity $\gamma_*$ at $z = 0$}
  \end{minipage}%
  \begin{minipage}{.5\textwidth}
    \centering
    \includegraphics[scale = 0.55]{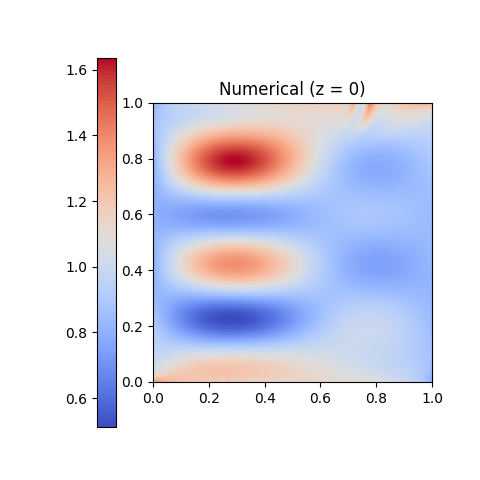}
    \caption{Reconstructed conductivity at $z = 0$}
  \end{minipage}
  \caption{A given by (\ref{D5}). Ran for 10 iterations on a mesh of size 100 x 100 x 10.}
  \label{figexample5}
\end{figure}

\vspace{0.5cm}

{\bf Example 5.}
Next we choose $A$ to be spatially varying as well as dependent on $\gamma$, as follows,
\begin{equation}\label{D5}
  A(\vec{x}, \gamma)\ \ =\ \ \begin{bmatrix}
  0.4(\gamma + 1)^{2} & 0.25(x^{2} + y^{2})\gamma& 0\\
  0.25(x^{2} + y^{2})\gamma & 3\gamma & 0\\
  0 & 0 & \gamma
  \end{bmatrix}
\end{equation}
and we let $\gamma_*$ be the trigonometric function
$$
  \gamma_*(x, y) = \sin(10x)\sin(5y)\sin(7(1 - x))\sin(y - 1) + 1.
$$
Numerical results are given in Figure \ref{figexample5}.

\vspace{0.5cm}

{\bf Example 6.}
Finally, we present an example in which $A$ is spatially varying as well as dependent on $\gamma$ and $\gamma_*$ is z-dependent . We let

\begin{equation} \label{D6}
  A(\vec{x}, \gamma)\ \ =\ \ \begin{bmatrix}
  0.4(\gamma + 1)^{2} & 0.25((x - 0.5)^{2} + (y - 0.5)^{2})\gamma& 0\\
  0.25((x - 0.5)^{2} + (y - 0.5)^{2})\gamma & 3\gamma & 0\\
  0 & 0 & \gamma
  \end{bmatrix}
\end{equation}
 and we let $\gamma_*$ be a piecewise constant corresponding to a spherical inclusion, that is,
 $$
  \gamma_*(x, y, z) = \begin{cases}
  2\ , & (x - 0.5)^{2} + (y - 0.5)^{2} + (z - 0.5)^{2} \leq 0.4\\
  1\ , & \text{else}
  \end{cases}.
$$
Slices of the three dimensional reconstruction are shown in Figures \ref{fig6-1} - \ref{fig6-5}.

\begin{figure}[htbp]
  \centering
 \begin{minipage}{.5\textwidth}
    \centering
    \includegraphics[scale = 0.5]{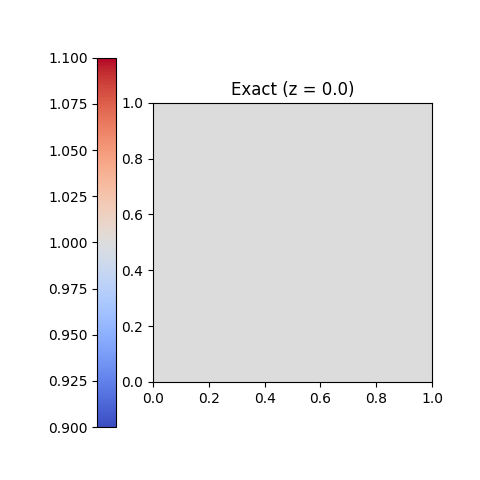}
  \end{minipage}%
  \begin{minipage}{.5\textwidth}
    \centering
    \includegraphics[scale = 0.5]{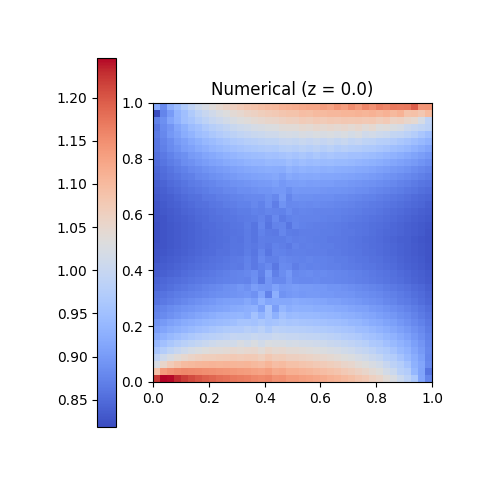}
  \end{minipage}
   \caption{Reconstruction for (\ref{D6}) at $z=0$.}
  \label{fig6-1}
\end{figure}

\begin{figure}[h!]
  \centering
  \begin{minipage}{.5\textwidth}
    \centering
    \includegraphics[scale = 0.5]{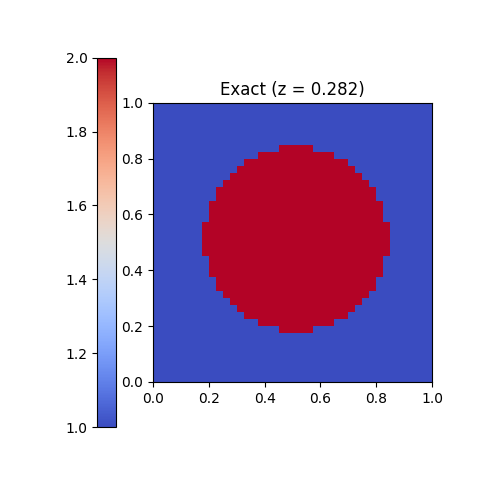}
  \end{minipage}%
  \begin{minipage}{.5\textwidth}
  \centering
   \includegraphics[scale = 0.5]{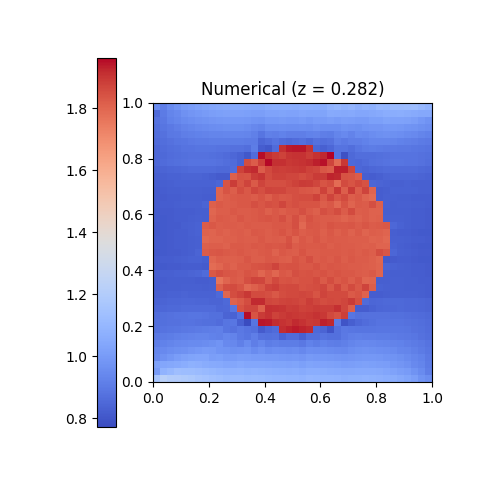}
  \end{minipage}
   \caption{Reconstruction for (\ref{D6}) at $z=0.282$.}
  \label{fig6-2}
\end{figure}

\begin{figure}[htbp]
 \centering
  \begin{minipage}{.5\textwidth}
    \centering
    \includegraphics[scale = 0.5]{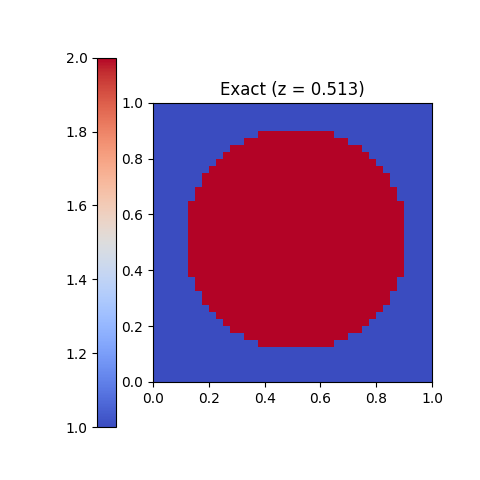}
 \end{minipage}%
  \begin{minipage}{.5\textwidth}
    \centering
    \includegraphics[scale = 0.5]{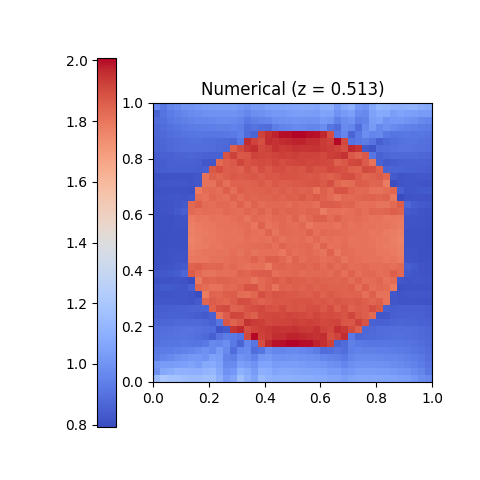}
  \end{minipage}
   \caption{Reconstruction for (\ref{D6}) at $z=0.513$.}
  \label{fig6-3}
  \end{figure}

  \begin{figure}[htbp]
 \centering
  \begin{minipage}{.5\textwidth}
    \centering
    \includegraphics[scale = 0.5]{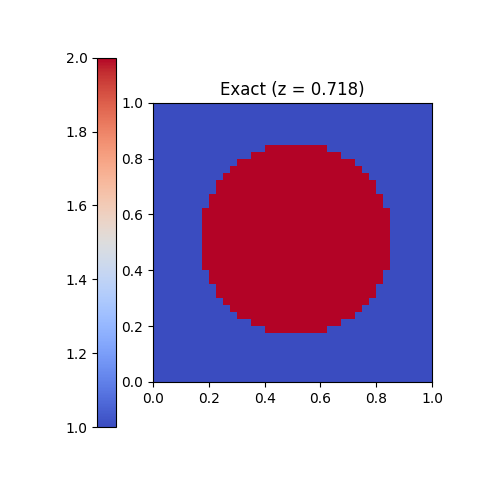}
  \end{minipage}%
    \begin{minipage}{.5\textwidth}
    \centering
    \includegraphics[scale = 0.5]{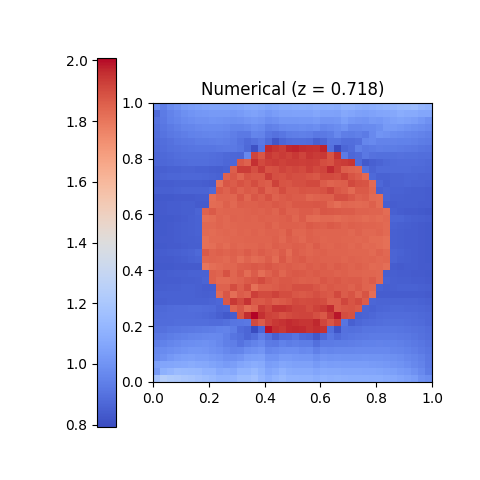}
  \end{minipage}
 \caption{Reconstruction for (\ref{D6}) at $z=0.718$.}
  \label{fig6-4}
\end{figure}

\begin{figure}[htbp]
  \centering
  \begin{minipage}{.5\textwidth}
    \centering
    \includegraphics[scale = 0.5]{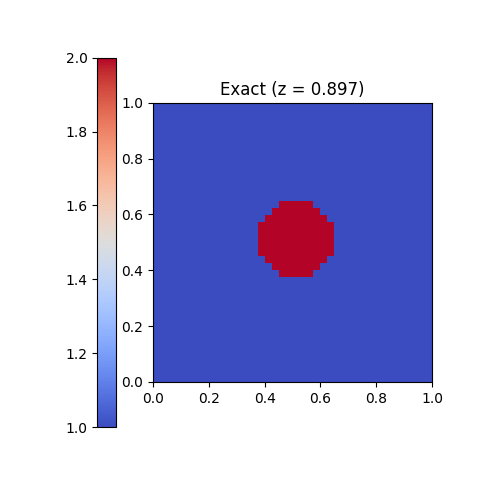}
  \end{minipage}%
  \begin{minipage}{.5\textwidth}
    \centering
    \includegraphics[scale = 0.5]{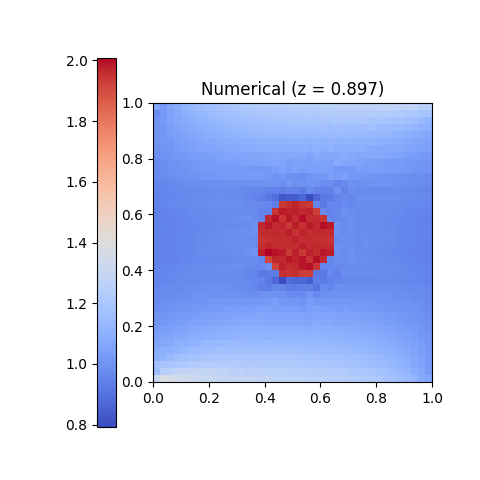}
  \end{minipage}
  \centering
  \caption{Reconstruction for $A$ given by the $z$-dependent (\ref{D6}) for   $z = 0.897$. Ran for 10 iterations on a mesh of size 40 x 40 x 40.}
  \label{fig6-5}
\end{figure}
\clearpage




\section{Conclusions}\label{conclusions}
In this work we studied issues of stability and reconstruction of the anisotropic conductivity $\sigma$ of a biological medium $\Omega\subset\mathbb{R}^3$  by the hybrid inverse problem of Magneto-Acoustic Tomography with Magnetic Induction MAT-MI. More specifically, we considered a class of conductivities which correspond to a one-parameter family of symmetric and uniformly positive matrix-valued functions $t\mapsto A(x, t)$, which are \textit{a-priori} known to depend nonlinearly on $t\in[\lambda^{-1}, \lambda]$. This gives rise to the family of anisotropic conductivities $A(x,\gamma(x))$, $x\in\Omega$, for which the goal is to stably reconstruct the scalar function $\gamma$ in $\Omega$. We showed that if $A$ belongs to certain classes of admissible anisotropic structures, then a Lipschitz type stability estimate of the scalar function $\gamma$ from the data given by and internal functional, holds true. In particular, the argument for our theoretical framework requires that $\gamma$ and $A$ belong to $C^{1,\beta}$. Our stability estimates extend the results in \cite{A-Q-S-Z} to the case where $\sigma$ depends nonlinearly on $\gamma$, hence allowing us to consider a more realistic type of anisotropic structures. Furthermore, we showed that the convergence of the reconstruction algorithm introduced in \cite{A-Q-S-Z} extends to this nonlinear case, and demonstrated its effectiveness in several numerical experiments.

Several questions remain, including, as mentioned in \cite{A-Q-S-Z}, the reconstruction of full anisotropy, which we expect will require more measurements at hand. It will also be interesting to investigate more precisely to what extent the regularity assumptions considered in this paper are needed for the convergence of the reconstruction algorithm in practice; and the related question of when one needs to actually regularize the iterates by invoking the projections into the convex set $\Tilde{S}$. These questions are the subject of future work.

\section*{Acknowledgments}
The research conducted by N. Donlon and R. Gaburro  in this publication was funded by the Irish Research Council under the Grant number: GOIPG/2021/527. R. Gaburro was partially supported by Science Foundation Ireland under Grant number 16/RC/3918. S. Moskow was partially supported by NSF grant DMS-2008441.



\end{document}